\documentclass[12pt]{article}

\usepackage[english]{babel}

\usepackage[letterpaper,top=2cm,bottom=2cm,left=3cm,right=3cm,marginparwidth=1.75cm]{geometry}
\usepackage{amsmath,amssymb,amsthm,amsfonts}
\usepackage{paralist}
\usepackage{subfigure}
\usepackage[title]{appendix}

\UseRawInputEncoding 

\usepackage{graphics} 
\usepackage{epsfig} 
\usepackage{graphicx}
\usepackage{caption}
\usepackage{epstopdf}
\usepackage[colorlinks=true]{hyperref}
\usepackage{float}
\usepackage{pdfpages}
\usepackage{amsfonts}
\usepackage{amsmath,amsthm}
\usepackage{multirow}
\usepackage[T1]{fontenc}
\usepackage{lmodern}

\usepackage{amsmath}
\usepackage{txfonts}
\usepackage{enumerate}
\usepackage[numbers,sort&compress]{natbib}
\usepackage[font=scriptsize,labelfont=bf,width=0.95\textwidth]{caption}

\numberwithin{equation}{section}

\newcommand{\be}{\begin{equation}}
\newcommand{\ee}{\end{equation}}

\newcommand{\ben}{\begin{eqnarray*}}
\newcommand{\enn}{\end{eqnarray*}}

\newtheorem{theorem}{\textbf Theorem}[section]
\newtheorem{lemma}{\textbf Lemma}[section]

 \numberwithin{equation}{section}
\newtheorem{remark}{Remark}[section]
\renewcommand{\theequation}{\arabic{section}.\arabic{equation}}

\usepackage{microtype}

\begin{document}

\title{\textbf{
Local
Minimizers in Second Order Mean-field Games Systems with Choquard Coupling}}
\author{
Fanze Kong \thanks{Department of Applied Mathematics, University of Washington, Seattle, WA 98195, USA; fzkong@uw.edu}, 
Yonghui Tong\thanks{Center for Mathematical Sciences, Wuhan University of Technology, Wuhan 430070, China; myyhtong@whut.edu.cn} and
Xiaoyu Zeng\thanks{Center for Mathematical Sciences, Wuhan University of Technology, Wuhan 430070, China; xyzeng@whut.edu.cn}
}
\date{\today}
\maketitle
\abstract{

Mean-field Games systems (MFGs) serve as paradigms to describe the games among a huge number of players.  In this paper, we consider the ergodic Mean-field Games systems in the bounded domain with Neumann boundary conditions and the decreasing Choquard coupling. Our results provide sufficient conditions for the existence of solutions to MFGs with Choquard-type coupling. 
 More specifically, in the mass-subcritical and critical regimes, the solutions are characterized as global minimizers of the associated energy functional. In the case of mass supercritical exponents, up to the Sobolev critical threshold, the solutions correspond to local minimizers.  The proof is based on variational methods, in which the regularization approximation, convex duality argument, elliptic regularity and Hardy-Littlewood-Sobolev inequality are comprehensively employed.

\medskip


\maketitle



\section{Introduction}\label{intro1}
This paper is devoted to the study of the stationary Mean-field Games systems in the bounded domain $\Omega\subset \mathbb R^n$, $n\geq 1$ with a smooth boundary $\partial\Omega$, which are
 \begin{align}\label{goalmodel}
\left\{\begin{array}{ll}
-\Delta u+C_H|\nabla u|^{\gamma}+\lambda=f(m),\  &\text{in} \ \Omega,\\
\vspace{0.5ex}
\Delta m+C_{H}\gamma\nabla\cdot (m|\nabla u|^{\gamma-2}\nabla u)=0,\  &\text{in} \ \Omega,\\
\vspace{0.5ex}
\frac{\partial u}{\partial\boldsymbol{n}}=0,\ \frac{\partial m}{\partial\boldsymbol{n}}+ C_{H}\gamma m|\nabla u|^{\gamma-2}\nabla u \cdot\boldsymbol{n} =0, \  &\text{on} \ \partial\Omega,\\
\vspace{0.5ex}
 \int_{\Omega}m\, dx=1,\ \ \int_{\Omega}u\, dx=0,
\end{array}
\right.
\end{align}
where triple $(m,u,\lambda)$ denotes a solution, $\gamma>1$, and $C_H>0$ is a constant.  Here $\textbf{n}$ is the unit outer normal and 
$f(m)$ represents the coupling and the decreasing coupling typically has the following two forms
\begin{align}\label{formfmtaken}
f(m)=-C_fm^{p}\text{ or } - C_{f} \left( K_{\alpha}\ast m \right),
\end{align}
where $p>1$, $C_f>0$ is a constant and $K_{\alpha}$ is defined as the following Riesz potential
\begin{align}\label{MFG-K}
K_{\alpha}=\frac{1}{|x|^{n-\alpha}}\text{ with }0<\alpha<n.
\end{align}

In a recent work, Cirant, Cosenza and Verzini \cite{Cirant2024CVPDE} showed the existence of local minimizers to \eqref{goalmodel} with the coupling $f(m)$ taken as the former form shown in \eqref{formfmtaken} under the mass supercritical exponent up to Sobolev critical exponent cases by using the variational method.  The main goal in this paper is to extend their results in the case of the nonlocal coupling. To achieve this, we need employ the regularization approximation, convex duality argument, elliptic regularity and Hardy-Littlewood-Sobolev inequality comprehensively .  More specifically, we consider $f(m)=-C_f(K_{\alpha}\ast m)$ shown in \eqref{formfmtaken} and study the existence of local minimizers to \eqref{goalmodel} with $\alpha\in[n-2\gamma',n-\gamma']$, where $\gamma'=\frac{\gamma}{\gamma-1}$ is the conjugate exponent of $\gamma$. 
\subsection{Mean-field Games Systems}\label{MFGIntroduction}
In 2007, Huang et al. \cite{Huang} and Lasry et al. \cite{Lasry} independently proposed a class of 
forward-backward parabolic PDE systems to describe the differential games among a huge number of agents.  The corresponding theories have potential applications in the field of economics, finance, management, etc. 
Now,  let us briefly review the formulation of ergodic Mean-field Games systems.  First of all, consider the trajectory of the $i$-th player as follows
\begin{align}\label{game-process-dXti}
dX_t^i=-\nu^i_t dt+\sqrt{2}dB_t^i, \ \ X_0^i=x^i\in\Omega,~i=1,\cdots,N,
\end{align}
where $\Omega$ could be the bounded domain or the whole space, $x^i$ denotes the initial condition, $\nu^i_t$ is the velocity and $B_t^i$ is the Brownian motion.   Assume Brownian motions $B_t^i$ are independent and all agents are homogeneous, then we have $X_t^i$, $i=1,\cdots,N$ follow the same process and drop "$i$" in \eqref{game-process-dXti} for the simplicity of notation.  The goal of each player is to minimize the following long-run cost
\begin{align}\label{longsenseexpectation}
J(\nu_t):=\limsup_{T\rightarrow+\infty}\inf_{\nu}\frac{1}{T}\int_0^T[L(\nu_t)+f(m(X_t))] dt,
\end{align}
where the Lagrangian function $L$ is the cost of moving with velocity $v_{t}$ and $f(m(X_t))$ is the cost of being at position $X_{t}$. By employing some standard dynamic programming principle, the following ergodic problem is formulated 
 \begin{equation}\label{MFG-SS}
\left\{
\begin{array}{ll}
-\Delta u+H(\nabla u)+\lambda=f(m) , &x \in \Omega,     \\
 \Delta m+\nabla\cdot (m\nabla H(\nabla u))=0,&x \in \Omega,\\
\int_{\Omega} mdx=M>0,
\end{array}
\right.
\end{equation}
where the triple $(m,u,\lambda)$ denotes the solution and $f$ is the coupling.  Here Hamiltonian $H$ is the Legendre duality of $L$, i.e. $H(\boldsymbol{p})=\sup_{\boldsymbol{v}\in \mathbf{R^n}}(\boldsymbol{v}\cdot \boldsymbol{p}-L(\boldsymbol{v}))$. 
 $H$ is in  general assumed to be convex. There is a vast literature associated with problem \eqref{MFG-SS}, see \cite{cirant2016stationary, bernardini2022mass, bernardini2023ergodic, meszaros2015variational, cesaroni2018concentration, CesaroniCirant2019, Francisco2018siam} and the references therein.

We also mention that the stationary Mean-field Games systems can be trivialized to nonlinear $\gamma'$-Laplacian Schr\"{o}dinger equations when $H$ is chosen as 
\begin{align}\label{Hamiltonian}
H(\boldsymbol{p}):=C_{H}\vert \boldsymbol{p}\vert^{\gamma}, \  \ \exists \gamma>1, \ \ C_{H}>0.
\end{align}
Indeed, Fokker-Planck equation in (\ref{MFG-SS}) can be reduced into the following form:
\begin{align}\label{FPeqpartially}
\nabla m+C_H\gamma m|\nabla u|^{\gamma-2}\nabla u=0~~\text{a.e.,}~~x\in \Omega.
\end{align}
Similarly as shown \cite{cirant2024critical}, we define $v:=m^{\frac{1}{\gamma'}}$ and obtain from (\ref{FPeqpartially}) and the $u$-equation in (\ref{MFG-SS}) that
\begin{align}\label{nonlinear-Schrodinger}
\left\{\begin{array}{ll}
-\mu\Delta_{\gamma'} v+[f(v^{\gamma'})-\lambda]v^{\gamma'-1}=0,~x\in\Omega,\\
\int_{\Omega} v^{\gamma'}\,dx=M,~v>0,~\mu=\big(\frac{\gamma'-1}{C_H}\big)^{\gamma'-1},
\end{array}
\right.
\end{align}
where $\Delta_{\gamma'}$ is the $\gamma'$-Laplacian and given by $\Delta_{\gamma'}v=\nabla\cdot(|\nabla v|^{\gamma'-2}\nabla v)$. Regarding the Schr\"{o}dinger equation involving $\gamma'$-Laplacian, it is well-known that with the consideration of Neumann boundary conditions, (\ref{nonlinear-Schrodinger}) admits the following variational structure:
\begin{align}\label{variation-schrodinger}
\mathcal{I}(v):=\int_{\Omega}\frac{\mu}{\gamma'}|\nabla v|^{\gamma'}+F(v)\, dx,
\end{align}
    where $F(v)$ denotes the anti-derivative of $f(v^{\gamma'})v^{\gamma'-1}.$  In particular, when $\gamma'=2$ and  $f(v^2)=-K_{\alpha}\ast v^{2}$ in (\ref{nonlinear-Schrodinger}), the equation becomes the standard nonlinear Schr\"{o}dinger equation with the Hartree-type aggregation term. 




Inspired by the relation between Schr\"{o}dinger equations and Mean-field Games systems discussed above, Cirant et al. \cite{cirant2024critical} studied the existence and blow-up behaviors of ground states to \eqref{MFG-SS} when $f$ has the former form given by \eqref{formfmtaken} under the mass critical exponent case.  Motivated by this work and \cite{bernardini2023ergodic}, Kong et al. \cite{KongFanze2024} proved the existence and asymptotic behaviors of ground states to \eqref{goalmodel} when $\alpha=n-\gamma'.$  In particular, Bernardini and Cesaroni \cite{bernardini2022mass} studied the mass subcritical exponent case with $\alpha\in(n-\gamma',n)$ via the variational method. We utilize some idea mentioned in \cite[Section 1]{bernardini2023ergodic} to study the Sobolev critical case (i.e. $\alpha=n-2\gamma'$) with non-local coupling. It is well-known that the solutions to system \eqref{goalmodel} can be obtained by finding the critical points of the following variational functional
\begin{align}\label{energy-dual}
\mathcal E(m,\boldsymbol{w}):=
\int_{\Omega} \left[mL\bigg(-\frac{\boldsymbol{w}}{m}\bigg)+F(m)\right]\, dx, \ \  (m,\boldsymbol{w}) \in  \mathcal{A},
\end{align}
where Lagrangian $L$ is defined by
\begin{align}\label{general-Lagrangian}
L\bigg(-\frac{\boldsymbol{w}}{m}\bigg):=\left\{\begin{array}{ll}
\sup\limits_{\boldsymbol{p}\in\mathbb R^n}\big(-\frac{\boldsymbol{p}\cdot \boldsymbol{w}}{m}-H(\boldsymbol{p})\big),&m>0,\\
0,&(m,\boldsymbol{w})=(0,\boldsymbol{0}),\\
+\infty,&\text{otherwise},
\end{array}
\right.
\end{align}
$F(m):=-\frac{1}{2}C_{f}(K_{\alpha}\ast m)m$ for $m\geq 0$ and $F(m)=0$ for $m\leq 0$.  Here the admissible set $\mathcal{A}$ is given by
\begin{align}\label{constraint-set-K}
\mathcal{A}:=\Big\{&(m,\boldsymbol{w})\in (L^{1}(\Omega)\cap W^{1,\beta}(\Omega))\times L^{\frac{\gamma' q_{\alpha}}{\gamma'+q_{\alpha}-1}}(\Omega)\nonumber\\
~&\text{s. t.}\int_{\Omega}\nabla m\cdot\nabla\varphi\,dx=\int_{\Omega}\boldsymbol{w}\cdot\nabla\varphi\, dx,\forall \varphi\in C_c^{\infty}(\Omega),\nonumber
~\int_{\Omega}m\,dx=1>0,~m\geq 0\text{~a.e.~}\Big\}, \\  & \text{with}\  \frac{1}{\beta}:=\frac{1}{\gamma'}+\frac{1}{\gamma q_{\alpha}}\ \text{and } \ q_{\alpha}=\frac{2n}{n+\alpha}. 
\end{align}
As shown in \cite{bernardini2023ergodic},  there are two critical exponents:
\begin{equation}\label{criticalexponents}
   \alpha_{mc}=\left\{\begin{array}{ll}
   n-\gamma',& n>\gamma', ~\text{(mass- critical)}, \\
   0, & n\leq \gamma',
\end{array}
\right.
\end{equation}
~\text{and}~
\begin{equation}
\alpha_{sc}=\left\{\begin{array}{ll}
n-2\gamma',& n>2\gamma',~ \text{(Sobolev- critical)},\\
0,&n\leq 2\gamma'.
\end{array}
\right.
\end{equation}

In fact, $\alpha_{mc}$ is the mass-critical exponent arising from the Gagliardo-Nirenberg inequality involving the $L^{1}$ norm of $m$ and $\alpha_{sc}$ defined as the Sobolev-critical exponent, is related to the Hardy-Littlewood-Sobolev inequality.  By using the notations discussed above, the range of exponent $\alpha$ in \eqref{goalmodel} can be classified as follows:
\begin{enumerate}
\setcounter{enumi}{0}
\item[$(H1)$]  Sobolev-supercritical:  $0<\alpha< \alpha_{sc}$;

\item[$(H2)$]  Sobolev-critical:  $\alpha=\alpha_{sc}$;

\item[$(H3)$]  mass-supercritical up to  Sobolev-subcritical: $\alpha_{sc}<\alpha<\alpha_{mc}$;

\item[$(H4)$] mass-critical: $\alpha=\alpha_{mc}$;
    
\item[$(H5)$] mass-subcritical: $\alpha_{mc}<\alpha<n$.
\end{enumerate}
Our goal in this paper is to prove the existence of solutions to (\ref{goalmodel}) with $f(m)=-C_f(K_{\alpha}\ast m)$ and they correspond to local minimizers of $\mathcal{E}$ given by (\ref{energy-dual}). Our contributions are to study the cases (H2) and (H3)  when the decreasing coupling is Choquard type. Plenty of work has been devoted to the Choquard-type Schr\"odinger equations {such as the existence, uniqueness, symmetry and multiplicity of the solutions and normalized solutions, see e.g. \cite{Schaftingen1,Schaftingen2,Tanaka1,Tanaka2,Lieb,Lions,Nils,Lions2} and the references therein}.  We mention that there are also some relevant  results for the existence of local minimizers in nonlinear Schr\"odinger equations with local couplings \cite{Noris2014anal.pde,Noris2019nonlinearity,Verzini2017cvpde}. 

The main results in this paper are summarized as follows:
\begin{theorem}\label{thm11-optimal}
Assume $0<\alpha<n$ and either
\begin{enumerate}
\setcounter{enumi}{0}
\item[$(\romannumeral 1)$]
      $\alpha_{mc}<\alpha<n$,  or

\item[$(\romannumeral 2)$]
	$\alpha=\alpha_{mc}>0$ and $C_{f}<C_{TV1}$,  or

\item[$(\romannumeral 3)$]
      $\alpha_{sc}<\alpha<\alpha_{mc}$ and $C_{f}< C_{TV2}$,  or
    
\item[$(\romannumeral 4)$]
      $\alpha=\alpha_{sc}>0$ and $C_{f}< C_{TV3}$,
\end{enumerate}
where $C_{TV1}$, $C_{TV2}$ and $C_{TV3}$ are given by \eqref{globalminregularized_1}, \eqref{lobalcriticlemma_1} and  \eqref{lobalcriticlemma_2} respectively. Then system \eqref{goalmodel} admits a solution $(u, m, \lambda)\in C^{2}(\overline{\Omega})\times W^{1,p}(\Omega)\times\mathbb R$ with $m>0$ for all $p\geq 1$. In addition, for case $(\romannumeral 1)$ and case $(\romannumeral 2)$, the pair $(m,\boldsymbol{w})$ with $\boldsymbol{w}=- C_{H}\gamma m|\nabla u|^{\gamma-2}\nabla u$ is a global minimizer of $\mathcal{E}$ in $\mathcal{A}$, while in case $(\romannumeral 3)$ and case $(\romannumeral 4)$, the pair $(m,\boldsymbol{w})$ is a local minimizer of $\mathcal{E}$ in $\mathcal{B}_{\bar{r}}$ with $\mathcal{B}_{\bar{r}}:=\bigg\{(m,\boldsymbol{w})\in \mathcal{A}:\|m\|_{L^{q_{\alpha}}(\Omega)}\leq \bar{r}\bigg\}$ for some $\bar{r}>0$.
\end{theorem} 

Theorem \ref{thm11-optimal} implies that solutions of \eqref{goalmodel} exist even under the mass-critical, mass-supercritical, and up to the Sobolev-critical cases. In particular, such solutions correspond to the minimizers of $\mathcal{E}$ in $\mathcal{A}$. As stated in Theorem \ref{thm11-optimal}, the mass-critical exponent $\alpha_{mc}$ is a threshold to obtain a global minimizer. Namely, for the mass-subcritical case, i.e. $\alpha>\alpha_{mc}$, $\mathcal{E}$ has a global minimizer, while for the case of $\alpha<\alpha_{mc}$, the energy $\mathcal{E}$ is not bounded from below and there only exist local minimizers in $\mathcal{A}$ under the further assumption that $C_{f}$ is small enough.  

The rest of this paper is organized as follows. In Section \ref{preliminaries}, we give some preliminary results that include the existence and regularity properties of Hamilton-Jacobi equations and Fokker-Planck equations.  Section \ref{sect3-optimal} is dedicated to the proof of Theorem \ref{thm11-optimal}. Throughout the paper, for simplicity we may assume that $\vert\Omega\vert=1$, 
due to our arguments still hold for the case of $\vert\Omega\vert\neq 1$.  Without confusing readers, $C>0$ is chosen as a generic constant, which may vary line by line. 

\medskip

\section{Preliminaries}\label{preliminaries}
In this section, we collect some preliminary results including estimates involving Riesz potential, existence and regularity properties of solutions to the Hamilton-Jacobi and Fokker-Planck equations.

\subsection{Estimates of Riesz Potential}\label{appendixA}
We state here some well-known  estimates on Riesz potential, see \cite[Theorem 4.3]{LiebLoss}, \cite[Theorem 14.37]{WheedenZygmund} and \cite[Theorem 2.8]{bernardini2023ergodic}.  First of all, we have

\begin{lemma}[Hardy Littlewood-Sobolev inequality] \label{H-L-S}
Assume that $0<\alpha<n$, $1<t<\frac{n}{\alpha}$ and $\frac{1}{s}=\frac{1}{t}-\frac{\alpha}{n}$. Then for any $f\in L^{t}(\mathbb R^n)$, 
\begin{align}\label{eqHLS_1}
	\Vert K_{\alpha}\ast f\Vert_{L^{s}(\mathbb R^n)}\leq C(n,\alpha,t) 	\Vert  f\Vert_{L^{t}(\mathbb R^n)},
\end{align}
where constant $C>0$ depending on $n$, $\alpha$ and $t$.  Moreover, suppose that $s,t>1$ with $\frac{1}{s}-\frac{\alpha}{n}+\frac{1}{t}=1$, $f\in L^{t}(\mathbb R^n)$ and $g\in L^{s}(\mathbb R^n)$. Then, we have there exists a sharp constant $C(n,\alpha,t)$ independent
of $f$ and $g$ such that
\begin{align}\label{eqHLS_2}
\bigg|\int_{\mathbb R^n}\int_{\mathbb R^n}\frac{ f(x) g(y)}{|x-y|^{n-\alpha}}\,dx\,dy\bigg|\leq C(n,\alpha,t)\Vert f\Vert_{L^{t}(\mathbb R^n)}\Vert g\Vert_{L^{s}(\mathbb R^n)}.
\end{align}
\end{lemma}
We remark that in Lemma \ref{H-L-S}, if $s=t$ in \eqref{eqHLS_2} and $f\in L^{\frac{2n}{n+\alpha}}(\mathbb R^n)$, then there exists a sharp constant $C(n, \alpha)$ independent
of $f$ and $g$ such that
\begin{align}\label{eqHLS_3}
	\bigg|\int_{\mathbb R^n}\int_{\mathbb R^n}\frac{f(x) f(y)}{|x-y|^{n-\alpha}}\,dx\,dy\bigg|\leq C(n,\alpha)\Vert f\Vert^{2}_{L^{\frac{2n}{n+\alpha}}(\mathbb R^n)}.
\end{align}

\begin{lemma}[C.f. Theorem 2.8 in \cite{bernardini2023ergodic}] \label{HolderforRiesz}
Let $0<\alpha<n$ and $1<t\leq +\infty$ be positive constants such that $t>\frac{n}{\alpha}$ and $s\in\big[1,\frac{n}{\alpha}\big)$ . Then for any $f\in L^{s}(\mathbb R^n)\cap L^{t}(\mathbb R^n)$, we have
\begin{align}\label{InfityRiesz}
	\Vert K_{\alpha}* f\Vert_{L^{\infty}(\mathbb R^n)}\leq C_1 \Vert  f\Vert_{L^{s}(\mathbb R^n)}+C_2 \Vert  f\Vert_{L^{t}(\mathbb R^n)},
\end{align}\label{holderRiesz}
where $C_1=C(n,\alpha,s)$ and $C_2=C(n,\alpha,t)$. Moreover, if $0<\alpha-\frac{n}{t}<1$, then we have 
\begin{align}\label{holderRiesz_1}
K_{\alpha}* f\in C^{0,\alpha-\frac{n}{t}}(\mathbb R^n).	
\end{align}
In particular, there exists constant $C:=C(n,\alpha, t)>0$ such that
\begin{align*}
\frac{\big|K_{\alpha}* f(x)-K_{\alpha}\ast f(y)\big|}{|x-y|^{\alpha-\frac{n}{t}}}\leq C \Vert  f\Vert_{L^{t}(\mathbb R^n)}, \ \  \ \ \ \forall  x \neq y.
\end{align*} 
\end{lemma}
In Lemma \ref{HolderforRiesz}, we establish $L^{\infty}$ and H{\"o}lder estimates of $K_{\alpha}* f$ under certain conditions on $f$ and $\alpha$.
{
\begin{remark}
 In  the subsequent analysis of this paper, $K_{\alpha}\ast f$ in a bounded domain $\Omega\subset \mathbb R^{n}$ is defined as $K_{\alpha}\ast (\chi_{\Omega} f)$ in the whole space. Here $\chi_{\Omega}(x)=1$ if $x\in \Omega$ and $\chi_{\Omega}(x)=0$ if $x\notin \Omega$.
\end{remark}
}
\subsection{Hamilton-Jacobi Equations}\label{subsection1}
Consider
\begin{align}\label{Hamilton-Jacobi-eq}
\left\{\begin{array}{ll}
-\Delta u+C_H|\nabla u|^{\gamma}+\lambda=f(x),\  &\text{in} \ \Omega,\\
\vspace{0.5ex}
\frac{\partial u}{\partial\boldsymbol{n}}=0, \  &\text{on} \ \partial\Omega,\\
\vspace{0.5ex} 
\int_{\Omega}u\, dx=0,
\end{array}
\right.
\end{align}
where $\Omega$ is a bounded domain with the smooth boundary $\partial\Omega$, $C_H>0$, $\gamma>1$ and $\lambda\in\mathbb R$. Then we have the following existence results 
\begin{lemma}\label{HJBexists}
    If $f\in C^{0,\theta}(\overline{\Omega})$ with some $0<\theta<1$, then there exists a unique constant $\lambda\in \mathbb R$ such that problem \eqref{Hamilton-Jacobi-eq} admits a unique solution in $C^{2,\theta}(\overline{\Omega})$ and
    \begin{equation}\label{FPE2.1}
        \lambda=\sup \Big\{c\in \mathbb R\  s.t. \ \exists \ u \in C^{2}(\overline{\Omega}), \frac{\partial u}{\partial{\boldsymbol{n}}}=0 \ \ \text{on}\ \partial\Omega: -\Delta u+C_H|\nabla u|^{\gamma}+c\leq f \Big\}.
    \end{equation}
    In addition,
    $$\|\nabla u\|_{L^{\infty}(\Omega)}\leq C_1=C_1\left(n,\Omega,\|f\|_{L^{\infty}(\Omega)}\right)\ \text{and} \ \|u\|_{C^{2,\alpha}(\overline{\Omega})}\leq C_2=C_2\left(n,\Omega,\|f\|_{C^{0,\alpha}(\overline{\Omega})}\right)$$
\end{lemma}
 \begin{proof}
 We refer the reader to the proof of Theorem 2.7 in \cite{Cirant2024CVPDE}.
 \end{proof}

For a-priori estimates of $u$, thanks to the Leray-Schauder fixed point theorem,  we have the following lemma
\begin{lemma}\label{nablabarrier_1}
Assume all conditions in Lemma \ref{HJBexists} hold, then for all $\sigma\in(0,\sigma_0]$ with $\sigma_0=\sigma_0(n,\Omega,\gamma,\alpha)>0$,  there exists $M=M(\sigma)\rightarrow 0$ as $\sigma\rightarrow 0$   such that, if  $\|f\|_{L^{\frac{n}{\gamma'}}(\Omega)}\leq \sigma\leq\sigma_0$,   then
    \begin{equation}\label{FPEregularity_1}
        \|\nabla u\|_{L^{n(\gamma-1)}(\Omega)}<M.
    \end{equation}
\end{lemma}
\begin{proof}
    The proof of this lemma is similar to Theorem 2.10 in \cite{Cirant2024CVPDE}. 
\end{proof}

\subsection{Fokker-Planck Equations}\label{subsection2}
In this subsection, we focus on the following Fokker-Planck equations
 \begin{align}\label{sect2-FP-eq}
\left\{\begin{array}{ll}
\Delta m+\nabla\cdot (m \boldsymbol{b})=0,\  &\text{in} \ \Omega,\\
\vspace{0.5ex}
\frac{\partial m}{\partial\boldsymbol{n}}+ m \boldsymbol{b}\cdot\boldsymbol{n} =0, \  &\text{on} \ \partial\Omega,\\
\vspace{0.5ex}
 \int_{\Omega}m\,=1,
\end{array}
\right.
\end{align}
where $\boldsymbol{b}: \mathbb R \rightarrow \mathbb R^n$ with $\boldsymbol{b}\in L^{s}(\Omega; \mathbb R^n)$ for some $s>n$, and $m\in W^{1,2}(\Omega)$ denotes the solution.  Concerning the regularity properties of solutions to \eqref{sect2-FP-eq}, we have
\begin{lemma}\label{regularity-mw1p}
Assume that $\boldsymbol{b} \in L^{\infty}(\Omega;\mathbb R^n)$, then problem \eqref{sect2-FP-eq} admits a unique weak solution $m\in W^{1,p}(\Omega)$ for all $p>1$ with 
$$
\|m\|_{W^{1,p}(\Omega)} \leq C :=C(\|\boldsymbol{b} \|_{\infty},p,n,\Omega).
$$
In addition, $m\in C^{0,\theta}(\overline{\Omega})$ for every $\theta\in (0,1)$ and there exists $C =C(\|\boldsymbol{b} \|_{\infty},p,n,\Omega)>0$ such that 
$$
0<C^{-1}\leq m(x)\leq C, \ \ \text{for any} \ \ x\in \Omega.
$$
\end{lemma}
\begin{proof}
See \cite[Proposition 2]{CesaroniCirant2019} or \cite[Themorem 2.3]{Cirant2024CVPDE}.
\end{proof}

\begin{lemma}\label{regularity-mw1p-1}
Let $m \in L^{p}(\Omega)$ for $p>1$ and assume that 
\begin{equation}\label{2.6-1}
\bigg|\int_{\Omega}m\Delta \phi \,dx\bigg|\leq K \|\nabla \phi\|_{L^{p'}(\Omega)}, 
\end{equation}
for all $\phi\in C^{\infty}(\overline{\Omega})$, $\frac{\partial\phi}{\partial\boldsymbol{n}}=0$ and some $K>0$. Then, $m\in W^{1,p}(\Omega)$ and it holds that
\begin{equation}\label{2.6-2}
\|m\|_{W^{1,p}(\Omega)}\leq C_{\Lambda}\left(K+\|m\|_{L^{p}(\Omega)}\right),
\end{equation}
for some $C_{\Lambda}=C_{\Lambda}(n,\Omega,p)$. In addition, we have the following local estimate
\begin{equation}\label{2.6-3}
\|m\|_{W^{1,p}(B_{R}(0))}\leq C_{\Lambda}(K+\|m\|_{L^{p}(B_{2R}(0))}),
\end{equation}
where $B_{2R}(0)\subset \Omega.$
\end{lemma}
\begin{proof}
    See \cite[Proposition 2.4]{Cirant2024CVPDE}.
\end{proof}

Lemma \ref{regularity-mw1p-1} indicates that for given $\boldsymbol{w}\in L^{p}(\Omega;\mathbb R^n)$, $p>1$, there exists $C>0$ such that $m$ satisfies
\begin{align}\label{FP-eq-main-w1p}
\|m\|_{W^{1,p}(\Omega)}\leq C\left(\|\boldsymbol{w}\|_{L^{p}(\Omega)}+\|m\|_{L^{p}(\Omega)}\right),
\end{align}
if $m\in W^{1,p}(\Omega)$ is defined as a weak solution of the following equation
\begin{align}\label{sect2-FP-eq-main}
\left\{\begin{array}{ll}
\Delta m+\nabla\cdot\boldsymbol{w}=0,\  &\text{in} \ \Omega,\\
\vspace{0.5ex}
\frac{\partial m}{\partial\boldsymbol{n}}+ \boldsymbol{w}\cdot\boldsymbol{n} =0,\  &\text{on} \ \partial\Omega,\\
\vspace{0.5ex}
 \int_{\Omega}m\,=1.
\end{array}
\right.
\end{align}
Inequality \eqref{FP-eq-main-w1p} exhibits the gradient estimates of solution $m$ to equation \eqref{sect2-FP-eq-main}. In order to tackle the existence of solution to (\ref{goalmodel}) with the Choquard coupling and Sobolev-critical exponent, we shall give a sharper version of \eqref{FP-eq-main-w1p}, which requires the $L^{n}$ norm of $\boldsymbol{b}$ to be controlled by some certain constants. More precisely, we have 
\begin{lemma}\label{b-infty}  
    Assume that $m\in W^{1,p}(\Omega)$ is a solution to (\ref{sect2-FP-eq}) and $\boldsymbol{b}\in L^{n} (\Omega;\mathbb R^n)$ with $\|\boldsymbol{b}\|_{L^{n}(\Omega)}\leq \zeta$.  
    Then for any $p<n$, the solution $m$ of \eqref{sect2-FP-eq} satisfies
    \begin{equation}\label{2.6.7}
        \|m\|_{W^{1,p}(\Omega)}\leq C=C(p,n,\Omega,\zeta).
    \end{equation}
\end{lemma}
\begin{proof}
    See \cite[Theorem 7.1 and 8.1] {AgmonShmue} or \cite[Proposition 2.6]{Cirant2024CVPDE} for the detailed proof.
\end{proof}

With the help of the conclusions shown in \eqref{FP-eq-main-w1p} and Lemma \ref{b-infty}, we can obtain the refined estimates of $m$ and the results are summarized as
\begin{lemma}\label{lemma21-crucial}
Assume that $(m,\boldsymbol{w})\in \left(L^1(\Omega)\cap W^{1,\beta}(\Omega)\right)\times L^1(\Omega)$ is a solution to \eqref{sect2-FP-eq-main} and 
\begin{equation}\label{lemmacrucialnewest}
\Lambda:=\int_{\Omega}|m|\Big|\frac{\boldsymbol{w}}{m}\Big|^{\gamma'}\,dx<\infty,
\end{equation}
where $\beta$ is given by
\eqref{constraint-set-K}.  Then, there exists $C_{\alpha}>0$ such that
\begin{equation}\label{lemmacrucialnewest-1}
\|m\|_{W^{1,\beta}(\Omega)},\ \|m\|_{L^{q_{\alpha}}(\Omega)} \leq C_{\alpha}\left(\Lambda+1\right).
\end{equation}
In addition, if $\alpha>\alpha_{mc}$, there exists $\delta>0$ such that 
\begin{equation}\label{lemmacrucialnewest-2}
\|m\|_{L^{q_{\alpha}}(\Omega)}^{q_{\alpha}(1+\delta)}\leq C_{\alpha}\left(\Lambda+1\right).
\end{equation}
 and if $\alpha=\alpha_{mc}$, then \eqref{lemmacrucialnewest-2} holds with $\delta=\frac{n-\gamma'}{n}$ and $q_{\alpha}(1+\delta)=2$.
\end{lemma}
\begin{proof}
Proceeding the similar argument shown in \cite[Proposition 5]{CesaroniCirant2019}, we can prove this lemma. For the 
sake of completeness, we exhibit the proof briefly.  First of all, we have the $m$-equation in \eqref{sect2-FP-eq-main} implies
\begin{align}\label{2.7-1}
   \Big{|}\int_{\Omega}m(-\Delta\phi)\,dx\Big{|} &=\Big{|}\int_{\Omega}\boldsymbol{w}\cdot\nabla\phi\,dx\Big{|}\leq \int_{\Omega}\left(|m|\Big|\frac{\boldsymbol{w}}{m}\Big|^{\gamma'}\right)^{\frac{1}{\gamma'}}|m|^{\frac{1}{\gamma}}|\nabla\phi|\,dx\nonumber\\
&\leq \left(\int_{\Omega}|m|\Big|\frac{\boldsymbol{w}}{m}\Big|^{\gamma'}\,dx \right)^{\frac{1}{\gamma'}}\|m\|_{L^{q_{\alpha}}(\Omega)}^{\frac{1}{\gamma}}\|\nabla\phi\|_{L^{\beta'}(\Omega)}\nonumber\\
&\leq \Lambda^{\frac{1}{\gamma'}}\|m\|_{L^{q_{\alpha}}(\Omega)}^{\frac{1}{\gamma}}\|\nabla\phi\|_{L^{\beta'}(\Omega)},
\end{align}
for any $\phi\in C^{\infty}(\overline{\Omega})$, where $\beta'=\frac{\beta}{\beta-1}$ and $\beta$ is given by
\eqref{constraint-set-K}. Then, we deduce from \eqref{2.7-1} and Lemma \ref{regularity-mw1p-1} that
\begin{equation}\label{2.7-2}
\|\nabla m\|_{L^{\beta}(\Omega)}\leq C_{\alpha}\left( 
 \Lambda^{\frac{1}{\gamma'}}\|m\|_{L^{q_{\alpha}}(\Omega)}^{\frac{1}{\gamma}}+\|m\|_{L^{\beta}(\Omega)}\right),
\end{equation}
where and in the discussion below, $C_\alpha>0$ denotes a constant dependent of $\alpha$, which may change line to line.  In addition, it follows from the Sobolev interpolation inequality that 
\begin{equation}\label{2.7-3}
 \|m\|_{L^{\beta}(\Omega)}\leq \|m\|_{L^{q_{\alpha}}(\Omega)}^{\frac{1}{\gamma}}\|m\|_{L^{1}(\Omega)}^{\frac{1}{\gamma'}}=\|m\|_{L^{q_{\alpha}}(\Omega)}^{\frac{1}{\gamma}}.
\end{equation}
 Therefore, by \eqref{2.7-2} and \eqref{2.7-3}, one has 
\begin{equation}\label{2.7-4}
    \|m\|_{W^{1,\beta}(\Omega)}\leq C_{\alpha}\left(\Lambda^{\frac{1}{\gamma'}}+1\right)\|m\|_{L^{q_{\alpha}}(\Omega)}^{\frac{1}{\gamma}}.
\end{equation}
On the other hand, it follows from the Sobolev embedding theorem that there exists a constant $C>0$ depends on $n$ and $\alpha$ such that $\|m\|_{L^{q}(\Omega)}\leq C\|m\|_{W^{1,\beta}(\Omega)} $, where we have used the fact $q_\alpha< \beta^*$ for any $\alpha>\alpha_{sc}$,  with $\beta^*=\frac{n\beta}{n-\beta}$ if $n>\beta$ and $\beta^*=\infty$ if $n\leq\beta$. Thus, one finds 
$$
\|m\|_{L^{q_{\alpha}}(\Omega)}\leq C_{\alpha}\left(\Lambda+1\right),
$$
which, together with \eqref{2.7-4}, yields $$\|m\|_{W^{1,\beta}(\Omega)}\leq C_{\alpha}\left(\Lambda+1\right).$$

It remains to prove \eqref{lemmacrucialnewest-2}. Since $1<q_{\alpha}<\beta^*$ for all $\alpha>\alpha_{mc}>\alpha_{sc}$, we invoke Sobolev interpolation to get
\begin{equation}\label{2.7-5}
    \|m\|_{L^{q_{\alpha}}(\Omega)}\leq\|m\|_{L^{1}(\Omega)}^{1-\theta}\|m\|_{L^{\beta^*}(\Omega)}^{\theta}\leq C_{\alpha}\left(\Lambda^{\frac{\theta}{\gamma'}}+1\right)\|m\|_{L^{q_{\alpha}}(\Omega)}^{\frac{\theta}{\gamma}}, 
\end{equation}
where $\theta$ satisfies $\frac{1}{q_{\alpha}}=\frac{\theta}{\beta^*}+1-\theta$. Then, we obtain that
\begin{equation}\label{2.7-6}
    \|m\|^{q_{\alpha}(1+\delta)}_{L^{q_{\alpha}}(\Omega)}\leq C_{\alpha}\left(\Lambda+1\right)\|m\|_{L^{q_{\alpha}}(\Omega)}, 
\end{equation}
where $\delta=\frac{1}{q_{\alpha}-1}\left(\frac{\gamma'+n}{n}-q_{\alpha}\right)$. In particular, if $\alpha =\alpha_{mc}$, one has $\delta=\frac{n-\gamma'}{n}$ and $q_{\alpha}(1+\delta)=2$.  Now, we finish the proof of this lemma.
\end{proof} 

\section{Existence of Global and Local Minimizers}\label{sect3-optimal}
In this section, we shall discuss the existence of minimizers of the energy functional associated with problem \eqref{goalmodel} and prove Theorem \ref{thm11-optimal}.  To this end, we first consider the regularized problem with smoothing couplings and show the existence of global or local minimizers of the regularized energy functional. Then, we will prove the existence of a solution to the regularized system by applying the convex duality 
argument.  Finally, by showing the uniform boundedness of the solution sequences, we can obtain the existence of the solution to the original problem after taking the limit.  To begin with, we focus on the following regularized system 
 \begin{align}\label{regularizedgoalmodel}
\left\{\begin{array}{ll}
-\Delta u+C_H|\nabla u|^{\gamma}+\lambda= -C_{f}\left( K_{\alpha}\ast m \ast \eta_{\varepsilon}\right),\  &\text{in} \ \Omega,\\
\vspace{0.5ex}
\Delta m+C_{H}\gamma\nabla\cdot (m|\nabla u|^{\gamma-2}\nabla u)=0,\  &\text{in} \ \Omega,\\
\vspace{0.5ex}
\frac{\partial u}{\partial\boldsymbol{n}}=0,\ \frac{\partial m}{\partial\boldsymbol{n}}+ C_{H}\gamma m|\nabla u|^{\gamma-2}\nabla u \cdot\boldsymbol{n} =0, \  &\text{on} \ \partial\Omega,\\
\vspace{0.5ex}
 \int_{\Omega}m\,=1,\ \ \int_{\Omega}u\,=0,
\end{array}
\right.
\end{align}
where $\eta_{\varepsilon}>0$ is the standard mollifier satisfying
$$
\int_{\Omega}\eta_{\varepsilon} \,dx=1,\ \ supp(\eta_{\varepsilon})\subset B_{\varepsilon}(0)\ \text{for} \ \varepsilon>0 \ \text{sufficiently small}.
$$

It is well known that solutions to auxiliary problem \eqref{regularizedgoalmodel} correspond to the critical points of the following energy functional 
\begin{align}\label{NengliangofApprop}
\mathcal E_{\varepsilon}(m,\boldsymbol{w})=C_L\int_{\Omega}\Big|\frac{\boldsymbol{w}}{m}\Big|^{\gamma'}m\,dx-\frac{1}{2}\int_{\Omega}C_{f}\bigg\{m(K_{\alpha}*m)\bigg\}\ast \eta_{\varepsilon}\,dx,\ \text{for} \ (m,\boldsymbol{w})\in \mathcal{A},
\end{align}
where $K_{\alpha}$ is defined in \eqref{MFG-K}. 
 Invoking the Young's inequality for convolution, the properties of mollifiers and Hardy-Littlewood-Sobolev inequality \eqref{eqHLS_3}, one finds that
\begin{equation}\label{Control_of_F}
    \left|\int_{\Omega}C_{f}\Big\{m(K_{\alpha}*m)\Big\}\ast \eta_{\varepsilon}\,dx\right|\leq C_{f} C(n,\alpha)\Vert m\Vert^{2}_{L^{\frac{2n}{n+\alpha}}(\Omega)}=C_{f} C(n,\alpha)\Vert m\Vert^{2}_{L^{q_{\alpha}}(\Omega)},
\end{equation}
where $C(n,\alpha)$ is the best constant of the inequality.

\subsection{Existence of Minimizers to Regularized Problem}
This subsection is devoted to the analysis of the regularized system \eqref{regularizedgoalmodel}.  In detail, we shall first prove the existence of global minimizers to (\ref{NengliangofApprop}) when $\alpha>\alpha_{mc}$ and $\alpha=\alpha_{mc}$.  Then with the additional procedures, we will study the existence of local minmizers when $\alpha<\alpha_{mc}$.   
There exist some results  analyzing the local minima in distinct systems under the mass-supercritical regime , see e.g., \cite{Jeanjean2016siam, Bellazzini2018mathann, Verzini2017cvpde}.  Motivated by their work, we study the existence of local minimizers in (\ref{NengliangofApprop}) and define
\begin{equation}\label{RegularProblem}
e_{r}:=\inf_{(m,\boldsymbol{w})\in \mathcal{B}_r}\mathcal E_{\varepsilon}(m,\boldsymbol{w})
\end{equation}
and 
\begin{equation}\label{RegularProblem_0}
\hat{e}_{r}:=\inf_{(m,\boldsymbol{w})\in \mathcal{U}_r}\mathcal E_{\varepsilon}(m,\boldsymbol{w}), 
\end{equation}
where
\begin{equation}\label{RegularProblem_1}
\mathcal{B}_r:=\bigg\{(m,\boldsymbol{w})\in \mathcal{A}:\|m\|_{L^{q_{\alpha}}(\Omega)}\leq r\bigg\}
\ \ \text{and }\ \
\mathcal{U}_r:=\bigg\{(m,\boldsymbol{w})\in \mathcal{A}:\|m\|_{L^{q_{\alpha}}(\Omega)}= r\bigg\},
\end{equation}
in which $q_{\alpha}$ is given by \eqref{constraint-set-K}.

Next, we show the existence of a minimizer to $\mathcal{E}_{\varepsilon}$ given by (\ref{NengliangofApprop}) in the set $\mathcal{B}_{r}$.
With some assumptions imposed on $C_f$ defined in \eqref{formfmtaken}, we also analyze the existence of a global minimum under the mass-subcritical and mass-critical cases. 
\begin{lemma}\label{globalminregularized}
   For any $r\geq1$, then $ e_{r}$ given by (\ref{RegularProblem}) has at least one minimizer.  Moreover, if $\alpha_{mc}<\alpha<n$ or $\alpha=\alpha_{mc}$ with 
 \begin{equation}\label{globalminregularized_1}
     C_{f}<C_{TV1}:=\frac{2C_{L}}{C_{\alpha}C(n,\alpha)},
 \end{equation}  
where $C_\alpha$ and $C(n,\alpha)$ are given by \eqref{lemmacrucialnewest-1}  and  \eqref{Control_of_F}, respectively.  Then  $\mathcal E_{\varepsilon}$ has a global minimum in $\mathcal{A}$. 
\end{lemma}
\begin{proof}
To begin with, we establish the lower bound of the energy $\mathcal E_{\varepsilon}$ in $\mathcal{B}_r$. 
 For any $(m,\boldsymbol{w})\in \mathcal{B}_r$, by using \eqref{lemmacrucialnewest-1} and \eqref{Control_of_F}, we obtain from the fact $\|m\|_{L^{q_{\alpha}}(\Omega)}\leq r$ in $\mathcal{B}_r$ that
\begin{equation}\label{global-1}
  \mathcal E_{\varepsilon}(m,\boldsymbol{w}) \geq \frac{C_{L}}{C_{\alpha}}\|m\|_{L^{q_{\alpha}}(\Omega)}-\frac{1}{2}C_{f}C(n,\alpha) \|m\|^{2}_{L^{q_{\alpha}}(\Omega)}-C_{L}\geq K_{0},
\end{equation} 
 for some constant $K_{0}$. Now, let $(m_k,\boldsymbol{w}_k)\in \mathcal{B}_r$ be a minimizing sequence of $e_{r}$, that is, $\lim\limits_{k\to \infty }\mathcal{E}_{\varepsilon}(m_k,\boldsymbol{w}_k)=e_{r}$. Then, one gets $\mathcal{E}_{\varepsilon}(m_k,\boldsymbol{w}_k)\leq e_{r}+1$. Moreover, thanks to \eqref{NengliangofApprop} and \eqref{Control_of_F}, we find
 \begin{align}\label{global-2}
\int_{\Omega}\Big|\frac{\boldsymbol{w_{k}}}{m_{k}}\Big|^{\gamma'}m_{k}\,dx&=C^{-1}_{L}\left(\mathcal{E}_{\varepsilon}(m_{k},\boldsymbol{w_{k}})+\frac{1}{2}\int_{\Omega}C_{f}\bigg\{m_{k}(K_{\alpha}*m_{k})\bigg\}\ast \eta_{\varepsilon}\,dx\right)\nonumber\\ 
&\leq C^{-1}_{L}\left(e_{r}+1+\frac{1}{2}C_{f}C(n,\alpha)r^2\right), 
\end{align}
 which shows that $\mathcal \int_{\Omega}\Big|\frac{\boldsymbol{w}_{k}}{m_{k}}\Big|^{\gamma'}m_{k}\,dx$ is bounded. Then, in light of \eqref{lemmacrucialnewest-1}, we have $\|m_k\|_{W^{1,\beta}(\Omega)}\leq C$ for some constant $C>0.$  Thus, we have up to a subsequence,
 \begin{equation}\label{global-3}
 m_k\rightharpoonup m\ \ \text{in}\ \ W^{1,\beta}(\Omega),\ \  m_k\to m\ \ \text{in}\ \ L^{1}(\Omega),\ \  
   m_k\to m\ \ \text{a.e.\ on}\ \ \Omega,
 \end{equation}
where $\beta$ is given by \eqref{constraint-set-K}. Invoking the H{\"o}lder inequality, one has
$$\int_{\Omega}|\boldsymbol{w}_{k}|^{\frac{\gamma'q_{\alpha}}{\gamma'+q_{\alpha}-1}}\leq \left(\int_{\Omega}\Big|\frac{\boldsymbol{w}_k}{m_{k}}\Big|^{\gamma'}m_{k}\,dx\right)^{\frac{q_{\alpha}}{\gamma'+q_{\alpha}-1}}\Vert m_{k}\Vert_{L^{q_{\alpha}}(\Omega)}^{\frac{\gamma'-1}{q_{\alpha}(\gamma'+q_{\alpha}-1)}},$$
which implies that $\boldsymbol{w}_k$ is equibounded in $L^{\frac{\gamma'q_{\alpha}}{\gamma'+q_{\alpha}-1}}(\Omega)$ and hence $\boldsymbol{w}_k\rightharpoonup \boldsymbol{w}\ \text{in}\ L^{\frac{\gamma'q_{\alpha}}{\gamma'+q_{\alpha}-1}}(\Omega)$. Thanks to \eqref{global-3} and Fatou's Lemma, we deduce that $\int_{\Omega}m=1$, $m\geq 0$ and $m\in \mathcal{B}_r$. In addition, one can infer that $(m,\boldsymbol{w})\in \mathcal{A}$. As a consquence, we have
$$
\mathcal{E}_{\varepsilon}(m,\boldsymbol{w})\leq \liminf_{k\to \infty} \int_{\Omega}\Big|\frac{\boldsymbol{w_{k}}}{m_{k}}\Big|^{\gamma'}m_{k}\,dx-\lim_{k\to \infty}\frac{1}{2}\int_{\Omega}C_{f}\bigg\{m_{k}(K_{\alpha}*m_{k})\bigg\}\ast \eta_{\varepsilon}\,dx\leq \liminf_{k\to\infty} \mathcal{E}_{\varepsilon}(m_{k},\boldsymbol{w}_{k}) = e_{r},
$$
 where the convexity  of $\int_{\Omega}\Big|\frac{\boldsymbol{w}}{m}\Big|^{\gamma'}m\,dx$ and Young's inequality for convolution  are used. 
 This indicates  that $e_{r}$ is attained by $(m,\boldsymbol{w})$.  
If $\alpha_{mc}<\alpha<n$, we utilize Lemma \ref{lemma21-crucial} to get
$$\|m\|_{L^{q_{\alpha}}(\Omega)}^{q_{\alpha}(1+\delta)}\leq C_{\alpha}\left(\Lambda+1\right),$$
where $\delta=\frac{1}{q_{\alpha}-1}\left(\frac{\gamma'+n}{n}-q_{\alpha}\right)$ and $\Lambda$ is defined by \eqref{lemmacrucialnewest}. Then, there holds for  some constant $K_{1}$ that 
\begin{equation}\label{global-4}
 \mathcal{E}_{\varepsilon}(m,\boldsymbol{w}) \geq \frac{C_{L}}{C_{\alpha}}\|m\|^{q_{\alpha}(1+\delta)}_{L^{q_{\alpha}}(\Omega)}-\frac{1}{2}C_{f}C(n,\alpha) \|m\|^{2}_{L^{q_{\alpha}}(\Omega)}-C_{L}\geq K_{1}, 
\end{equation} 
where we have used that $q_{\alpha}(1+\delta)>2$ for any $\alpha>\alpha_{mc}$. In addition, we define $$e:=\inf\limits_{(m,\boldsymbol{w})\in \mathcal{A}}\mathcal E_{\varepsilon}(m,\boldsymbol{w})$$
and choose a minimizing sequence $(m_k,\boldsymbol{w}_k)\in \mathcal{A}$ of $e$. Then, by a similar argument as used for the derivation of \eqref{global-2}, we get
\begin{equation}\label{global-5}
\int_{\Omega}\Big|\frac{\boldsymbol{w_{k}}}{m_{k}}\Big|^{\gamma'}m_{k}\, dx
    \leq C^{-1}_{L}\left(e+1+\frac{1}{2}C_{f}C(n,\alpha)\left(C_{\alpha}\left(1+\int_{\Omega}\Big|\frac{\boldsymbol{w_{k}}}{m_{k}}\Big|^{\gamma'}m_{k}\,dx\right)\right)^{\frac{2}{q_{\alpha}(1+\delta)}}\right),
\end{equation}
which, together with the fact $q_{\alpha}(1+\delta)>2$, shows that $\int_{\Omega}\Big|\frac{\boldsymbol{w_{k}}}{m_{k}}\Big|^{\gamma'}m_{k}\,dx$ is bounded.  Then by proceeding with the same argument shown above, we obtain the existence of a minimizer. 
 Finally, if $\alpha=\alpha_{mc}$, we find $q_{\alpha}(1+\delta)=2$ in (\ref{global-5}).  Then by choosing $\frac{2C_{L}}{C_{\alpha}C(n,\alpha)}>C_{f}$, one repeats the procedure shown above to obtain that $\mathcal{E}_{\varepsilon}$ has a global minimum. 
\end{proof}

We next turn our attention to the mass-supercritical and Sobolev-critical cases, that are, $\alpha_{sc}<\alpha<\alpha_{mc}$ and $\alpha=\alpha_{sc}$. We aim to show the existence of a local minimum, in which the crucial step is to show that $\hat{e}_{r}$ defined by (\ref{RegularProblem_0}) is achieved for some $r$, possibly on $\mathcal{B}_{r}\subset\mathcal{U}_{r}$. To this end, we shall find $\bar{r}$ such that $ e_{\bar{r}}<\hat{e}_{\bar{r}}$, which is shown as follows.
\begin{lemma}\label{lobalcriticlemma}      
If $\alpha_{sc}<\alpha <\alpha_{mc}$, assume that $\bar{r}=\frac{C_{L}}{C_{\alpha}C_{f}C(n,\alpha)}$ with $C_{\alpha}$ and $C(n,\alpha)$ defined in (\ref{lemmacrucialnewest-1}) and (\ref{Control_of_F}), also
\begin{equation}\label{lobalcriticlemma_1}
    C_{f}<C_{TV2}:=\min\left\{\frac{C_{L}}{C_{\alpha}C(n,\alpha)},\  \frac{C_{L}}{2C^{2}_{\alpha}C(n,\alpha)}\right\},\ 
\end{equation}
 or if $\alpha=\alpha_{sc}$, assume that $\bar{r}={C_{\alpha}}+1$ with
\begin{equation}\label{lobalcriticlemma_2}
    C_{f}<C_{TV3}:=\frac{2C_{L}}{C_{\alpha}C(n,\alpha)(1+C_{\alpha})^{2}}.
\end{equation}   
Then, $\mathcal{E}_{\varepsilon}$
admits a local minimum $(m_{\varepsilon},\boldsymbol{w}_{\varepsilon})\in \mathcal{A}\cap \mathcal{B}_{\bar{r}}$. In addition,
    $$\mathcal{E}_{\varepsilon}(m_{\varepsilon},\boldsymbol{w}_{\varepsilon})=e_{\bar{r}-\sigma}=e_{\bar{r}}$$
    for any $0<\sigma\leq \sigma_{0}$ with some $\sigma_{0}>0$ sufficiently small.
\end{lemma}
\begin{proof}
Firstly, we consider the case of $\alpha_{sc}<\alpha <\alpha_{mc}$.  Following the argument shown in \cite{Cirant2024CVPDE}, we shall find $r_{1}<r_{2}$ such that $\hat{e}_{r_{1}}<\hat{e}_{r_{2}}$.  With this, one has
  \begin{equation}\label{local-1}
  e_{r_2}=\min_{r\in[0,r_{2}]}\left\{\hat{e}_{r}\right\}\leq  \hat{e}_{r_{1}}<\hat{e}_{r_{2}},
   \end{equation}
  which implies the existence of an interior minimum of $\mathcal{E}_{\varepsilon}$ in $\mathcal{B}_{r_{2}}$. To achieve this, we choose $r_{1}=1$ and $r_{2}=\bar{r}=\frac{C_{L}}{C_{\alpha}C_{f}C(n,\alpha)}$, respectively. Noting that $ C_{f}<\frac{C_{L}}{C_{\alpha}C(n,\alpha)}$, we have $\bar{r}>1$. For any $(m,\boldsymbol{w})\in\mathcal{U}_{r_{1}}$, we apply H{\"o}lder's inequality then obtain from the assumption $\vert \Omega\vert=1$ that
  $$1=\Vert m\Vert_{L^{1}(\Omega)}\leq \left(\int_{\Omega}m^{q_{\alpha}}\, dx \right)^{\frac{1}{q_{\alpha}}} |\Omega|^{\frac{1}{q'_{\alpha}}}=1,$$
 which implies that $m(x)=1$ a.e. in $\Omega$ since the equality here holds. Moreover, it is easy to check that $(m,\boldsymbol{w})\equiv (1,\boldsymbol{0})\in \mathcal{U}_{r_{1}}$. Hence, in light of \eqref{NengliangofApprop}, one gets
  \begin{equation}\label{local-2}
      \hat{e}_{r_{1}}\leq \mathcal{E}_{\varepsilon}(1,\boldsymbol{0}) \leq 0.
  \end{equation}
  On the other hand, we derive from \eqref{global-1} with $r=\bar{r}$ that
 \begin{equation}\label{local-3}
       \hat{e}_{\bar{r}}\geq \frac{C_{L}}{C_{\alpha}}\bar{r}-\frac{1}{2}C_{f}C(n,\alpha)\bar{r}^{2}-C_{L}.
 \end{equation}
 To guarantee $\hat{e}_{r_{1}}<\hat{e}_{r_{2}}$, we are going to prove that 
 \begin{equation}\label{local-4}
      C_{L}<\frac{C_{L}}{C_{\alpha}}\bar{r}-\frac{1}{2}C_{f}C(n,\alpha)\bar{r}^{2}.
 \end{equation}
Indeed, it is straightforward to verify that $\bar{r}=\frac{C_{L}}{C_{\alpha}C_{f}C(n,\alpha)}$ is the maximum point of function $\varphi(r):=\frac{C_{L}}{C_{\alpha}}r-\frac{1}{2}C_{f}C(n,\alpha)r^{2}$. Hence, we conclude that \eqref{local-4} holds under the 
condition
 $$
    C_{f}<\min\left\{\frac{C_{L}}{C_{\alpha}C(n,\alpha)},\  \frac{C_{L}}{2C^{2}_{\alpha}C(n,\alpha)}\right\}.
$$
As a result, we obtain the desired conclusion.

Now, we consider the case of $\alpha=\alpha_{sc}$.  Similarly as the argument shown above, we apply again \eqref{local-3} and \eqref{local-2} to get  
\begin{equation}\label{local-5}
     C_{L}<\frac{C_{L}}{C_{\alpha}}\bar{r}-\frac{1}{2}C_{f}C(n,\alpha)\bar{r}^{2},
 \end{equation}
where we take $\bar{r}={C_{\alpha}}+1$ and have used the condition $C_{f}<\frac{2C_{L}}{C_{\alpha}C(n,\alpha)(1+C_{\alpha})^{2}}$.  Then it follows that  $\hat{e}_{r_{1}}<\hat{e}_{r_{2}}$.

Finally, we deduce from the continuity of $\varphi(r)$ that there exists $\sigma_{0}>0$ small enough such that $\hat{e}_{r_{1}}<\hat{e}_{\bar{r}-\sigma}$ for any $0<\sigma\leq \sigma_{0}$.  This completes the proof of our lemma.
\end{proof}
\begin{remark}
   As shown in the proof of Lemma \ref{lobalcriticlemma}, one can see that the choice of $\bar{r}$ depends on $C_{f}$ in the mass-supercritical case and the Sobolev-subcritical case.  As a consequence, $\bar r$ can be chosen large enough such that $\mathcal{E}_{\varepsilon}$ has an interior minimizer in $\mathcal{B}_{\bar{r}}$. Whereas, in the Sobolev-critical case, the construction of $\bar{r}$ is independent of $C_{f}$. 
\end{remark}
\subsection{Existence of Minimizers to Original Problem}
In this section, we shall finish the proof of Theorem \ref{thm11-optimal}.  First of all, we aim to obtain a solution to the regularized system \eqref{regularizedgoalmodel} by invoking the convex duality argument.  Noting that $\mathcal{E}_{\varepsilon}$ is not convex, we introduce the following convex functional 
\begin{equation}\label{linearizedfun}
\mathcal{J}_{\varepsilon}(m,\boldsymbol{w})=C_L\int_{\Omega}\Big|\frac{\boldsymbol{w}}{m}\Big|^{\gamma'}m\,dx-C_{f}\int_{\Omega}m(K_{\alpha}\ast m_{\varepsilon}\ast \eta_{\varepsilon})\,dx,
\end{equation}
where the second term is linear and $m_{\varepsilon}$ is given.  Next, we will show that $(m_{\varepsilon},\boldsymbol{w}_{\varepsilon})$, the minimizer of $\mathcal{E}_{\varepsilon}$, is also a minimizer of $\mathcal{J}_{\varepsilon}$. Moreover, we obtain the existence of $u_{\varepsilon}$ and $\lambda_{\varepsilon}$ by using Lemma \ref{HJBexists}, then prove $(m_{\varepsilon},u_{\varepsilon},\lambda_{\varepsilon})$ is a solution to  \eqref{regularizedgoalmodel}.
\begin{lemma}\label{regularizationsameminimizerl}  
   Let $(m_{\varepsilon},\boldsymbol{w}_{\varepsilon})$ be a global minimizer of $\mathcal{E}_{\varepsilon}$ in $\mathcal{A}$ or a local minimizer in $\mathcal{B}_{\bar{r}}$ obtained in Lemma \ref{globalminregularized} or Lemma \ref{lobalcriticlemma}. Then
\begin{equation}\label{sameminimizer}
\min_{(m,\boldsymbol{w})\in \mathcal{A}}\mathcal{J}_{\varepsilon}(m,\boldsymbol{w})=\mathcal{J}_{\varepsilon}(m_{\varepsilon},\boldsymbol{w}_{\varepsilon}).
\end{equation}
Moreover, $m_{\varepsilon}\in W^{1,p}(\Omega)$ for all $p>1$ and there exists $u_{\varepsilon}\in C^{2}(\Omega)$ and $\lambda_{\varepsilon}\in \mathbb R$ such that $(m_{\varepsilon},u_{\varepsilon},\lambda_{\varepsilon})$ solves system \eqref{regularizedgoalmodel} with
\begin{equation}\label{sameminimizer_11}
\boldsymbol{w}_{\varepsilon}=-C_{H}\gamma m_{\varepsilon}|\nabla u_{\varepsilon}|^{\gamma-2}\nabla u_{\varepsilon}, 
\end{equation}
and 
\begin{equation}\label{sameminimizer_22}
\Vert m_{\varepsilon}\Vert_{W^{1,\beta}(\Omega)}\leq C, \ \Vert m_{\varepsilon}\Vert_{L^{q_{\alpha}}(\Omega)}\leq C \ \text{and} \ \vert\lambda_{\varepsilon} \vert \leq C, \ \text{for some}\ C>0 \ \text{independent of}\ \varepsilon.
\end{equation} 
\end{lemma}
\begin{proof}
For any $(m,\boldsymbol{w})\in \mathcal{A}$ and $\mu\in(0,1)$, we define
\begin{equation}\label{sameminimizer_1}
m_{\mu}:=(1-\mu)m_{\varepsilon}+\mu m \ \ \text{and}\ \ \boldsymbol{w}_{\mu}:=(1-\mu)\boldsymbol{w}_{\varepsilon}+\mu \boldsymbol{w}. 
\end{equation}
When $(m_{\varepsilon},\boldsymbol{w}_{\varepsilon})$ is a global minimizer of $\mathcal{E}_{\varepsilon}$, 
we obtain that
\begin{equation}\label{sameminimizer_2}
C_L\int_{\Omega}\Big|\frac{\boldsymbol{w}_{\mu}}{m_{\mu }}\Big|^{\gamma'}m_{\mu}\,dx-C_L\int_{\Omega}\Big|\frac{\boldsymbol{w}_{\varepsilon}}{m_{\varepsilon}}\Big|^{\gamma'}m_{\varepsilon}\,dx\geq\frac{C_{f}}{2}\int_{\Omega}m_{\mu}(K_{\alpha}\ast m_{\mu}\ast \eta_{\varepsilon})\,dx-\frac{C_{f}}{2}\int_{\Omega}m_{\varepsilon}(K_{\alpha}\ast m_{\varepsilon}\ast \eta_{\varepsilon})\,dx,
\end{equation}
and by the convexity of $\int_{\Omega}\Big|\frac{\boldsymbol{w}}{m}\Big|^{\gamma'}m\,dx$, one has the following inequality holds
\begin{equation}\label{sameminimizer_3}
\mu\left(C_L\int_{\Omega}\Big|\frac{\boldsymbol{w}}{m}\Big|^{\gamma'}m\,dx-C_L\int_{\Omega}\Big|\frac{\boldsymbol{w}_{\varepsilon}}{m_{\varepsilon}}\Big|^{\gamma'}m_{\varepsilon}\,dx\right)\geq C_L\int_{\Omega}\Big|\frac{\boldsymbol{w}_{\mu}}{m_{\mu }}\Big|^{\gamma'}m_{\mu}\,dx-C_L\int_{\Omega}\Big|\frac{\boldsymbol{w}_{\varepsilon}} {m_{\varepsilon}}\Big|^{\gamma'}m_{\varepsilon}\,dx.
\end{equation}
Combining \eqref{sameminimizer_2} with \eqref{sameminimizer_3}, one finds
\begin{equation}\label{sameminimizer_4}
\mu\left(C_L\int_{\Omega}\Big|\frac{\boldsymbol{w}}{m}\Big|^{\gamma'}m\,dx-C_L\int_{\Omega}\Big|\frac{\boldsymbol{w}_{\varepsilon}}{m_{\varepsilon}}\Big|^{\gamma'}m_{\varepsilon}\,dx\right)\geq\mu C_{f}\int_{\Omega}\left(m(x)-m_{\varepsilon}(x)\right)(K_{\alpha}\ast m_{\varepsilon}\ast \eta_{\varepsilon})\,dx+O(\mu^2)
\end{equation}
Then, we divide the both hand sides of (\ref{sameminimizer_4}) by $\mu$  and take the limit $\mu\to 0$ to get
\begin{equation}\label{sameminimizer_5}
C_L\int_{\Omega}\Big|\frac{\boldsymbol{w}}{m}\Big|^{\gamma'}m\,dx-C_L\int_{\Omega}\Big|\frac{\boldsymbol{w}_{\varepsilon}}{m_{\varepsilon}}\Big|^{\gamma'}m_{\varepsilon}\,dx\geq C_{f}\int_{\Omega}\left(m(x)-m_{\varepsilon}(x)\right)(K_{\alpha}\ast m_{\varepsilon})\ast \eta_{\varepsilon}\, dx,
\end{equation}
which, together with the arbitrariness
of $(m,\boldsymbol{w})\in \mathcal{A}$, implies that  $(m_{\varepsilon},\boldsymbol{w}_{\varepsilon})$ minimizes $\mathcal{J}_{\varepsilon}$ in $\mathcal{A}$.

When $(m_{\varepsilon},\boldsymbol{w}_{\varepsilon})$ is a local minimizer of $\mathcal{E}_{\varepsilon}$, thanks to Lemma \ref{lobalcriticlemma}, one has $m_{\varepsilon}\in \mathcal{B}_{\bar{r}-\sigma}$ for some $\sigma>0$ small enough. As a consequence, we obtain that $m_{\mu}\in \mathcal{B}_{\bar{r}-\sigma}$ for $\mu>0$ sufficiently small, where 
$m_{\mu}$ is given by \eqref{sameminimizer_1}. Then, we repeat the argument shown above for the global minimizer and obtain the desired conclusion. Finally, we follow the arguments shown in \cite[Proposition 3.4]{bernardini2022mass} and \cite[Theorem 4]{cesaroni2019introduction} with the minor modification to derive the estimates stated in this lemma and the details are omitted.
\end{proof}

With the aid of Lemma \ref{globalminregularized}, \ref{lobalcriticlemma} and \ref{regularizationsameminimizerl}, we are able to show the conclusions stated in Theorem \ref{thm11-optimal}.  Here the crucial step is to establish an a-priori $L^{\infty}$ estimate satisfied by $m_{\varepsilon}$.  With this, one can pass the limit in $\varepsilon$ and show the existence of a classical solution to the original problem (\ref{goalmodel}).   We mention that $\alpha_{sc}=n-2\gamma'$ is the Hardy-Littlewood-Sobolev critical exponent and when $\alpha_{sc}<\alpha<n$, we can follow the argument shown in \cite{cesaroni2018concentration, Cirant2024CVPDE}} to derive an a-priori bounded of $m_{\varepsilon}$ in $L^{\infty}$ by blow-up analysis.  However, when $\alpha=\alpha_{sc}$,  some restrictive conditions must be imposed on $C_{f}$ due to the worse regularity of the nonlinear coupling in (\ref{goalmodel}).  We summarize the $L^\infty$ estimates of $m_{\varepsilon}$ as follows 
\begin{lemma}\label{mepsinfty}
    Let $(m_{\varepsilon},u_{\varepsilon},\lambda_{\varepsilon})$ be a solution of \eqref{regularizedgoalmodel} obtained in Lemma \ref{regularizationsameminimizerl}. Assume that $\alpha_{sc}<\alpha<n$ or $\alpha=\alpha_{sc}>0$ with $C_{f}\leq \frac{\sigma}{\bar{r}C_{n,\gamma'}}$ for some $\sigma>0$ small enough and $C_{n,\gamma'}>0$ depending only on $n$ and $\gamma'$, then there exists a constant $C>0$ independent of $\varepsilon$ such that
\begin{equation}\label{mvarepsilonLinfty} 
\Vert m_{\varepsilon}\Vert_{L^{\infty}(\Omega)}\leq C.
\end{equation}
\end{lemma}
\begin{proof}
We first consider the case of $\alpha_{sc}<\alpha<n$. To derive the $L^\infty$ estimate, we argue by contradiction and assume that 
$$
M_{\varepsilon}:=\max_{x\in \Omega}m_{\varepsilon}(x)=m_{\varepsilon}(x_{\varepsilon})\to +\infty.
$$ 
Noting that $\alpha>\alpha_{sc}$, 
let
\begin{equation}\label{eq3.4_0}
\varrho_{\varepsilon}:=M_{\varepsilon}^{-r},
\end{equation}
where $r>0$ (and thus $\varrho_{\varepsilon}\to 0$ as $\varepsilon\to 0$) to be determined later and $q_{\alpha}=\frac{2n}{n+\alpha}$. Obviously, we have $q_{\alpha}>\frac{n} {\alpha+\gamma'}$ due to $\alpha>\alpha_{sc}$. This ensures that the interval below is not empty and so we can fix $\bar q$ such that $$\bar q\in\bigg(\frac{n}{\alpha},\frac{nq_{\alpha}}{(n-\gamma'q_{\alpha})^+}\bigg),$$
where 
\begin{align*}
(n-\gamma'q_{\alpha})^+:=\left\{\begin{array}{ll}
n-\gamma'q_{\alpha},&\text{ if } n> \gamma'q_{\alpha},\\
0,&\text{ if } n\leq \gamma'q_{\alpha}.
\end{array}
\right.
\end{align*}
Now, we choose $r$ such that 
\begin{equation}\label{eq_r}
   \frac{1}{\gamma'}\left(1-\frac{q_{\alpha}}{\bar{q}} \right)< r<\frac{q_{\alpha}}{n}. 
\end{equation}
Moreover, we define
\begin{align}\label{eq3.4_1}
\left\{\begin{array}{ll}
\bar m_{\varepsilon}(x):=M_{\varepsilon}^{-1}m_{\varepsilon}(\varrho_{\varepsilon} x+ x_{\varepsilon}),\\
\bar u_{\varepsilon}(x):=\varrho_{\varepsilon}^{\frac{2-\gamma}{\gamma-1}} u_{\varepsilon} (\varrho_{\varepsilon}x+ x_{\varepsilon}).
\end{array}
\right.
\end{align}
It is easy to verify that $\bar m_{\varepsilon}(0)=1$ and $0\leq \bar m_{\varepsilon}(x)\leq 1$. Upon substituting \eqref{eq3.4_1} to \eqref{regularizedgoalmodel}, one has
\begin{align}\label{eq3.4_2}
\left\{\begin{array}{ll}
-\Delta \bar{u}_{\varepsilon}+C_H|\nabla \bar{u}_{\varepsilon}|^{\gamma}+\varrho_{\varepsilon}^{\gamma'}\lambda_{\varepsilon}=-C_{f}\varrho_{\varepsilon}^{\gamma'}\left( K_{\alpha}\ast {m}_{\varepsilon} \ast \eta_{{\varepsilon}}\right)(x_{\varepsilon}+\varrho_{\varepsilon} x),\  &\text{in} \ \Omega_{\varepsilon},\\
\vspace{0.5ex}
\Delta \bar{m}_{\varepsilon}+C_{H}\gamma\nabla\cdot (\bar{m}_{\varepsilon}|\nabla \bar{u}_{\varepsilon}|^{\gamma-2}\nabla \bar{u}_{\varepsilon})=0,\  &\text{in} \ \Omega_{\varepsilon},\\
\vspace{0.5ex}
\frac{\partial \bar{u}_{\varepsilon}}{\partial\boldsymbol{n}}=0,\ \frac{\partial \bar{m}_{\varepsilon}}{\partial\boldsymbol{n}}+ C_{H}\gamma \bar{m}_{\varepsilon}|\nabla \bar{u}_{\varepsilon}|^{\gamma-2}\nabla \bar{u}_{\varepsilon} \cdot\boldsymbol{n} =0, \  &\text{on} \ \partial\Omega_{\varepsilon},\\
\vspace{0.5ex}
\int_{\Omega_{\varepsilon}}\bar{m}_{\varepsilon}\,=M_{\varepsilon}^{-1}\varrho_{\varepsilon}^{-n},\ \ \int_{\Omega_{\varepsilon}}\bar{u}_{\varepsilon}\,=0,
\end{array}
\right.
\end{align}
where $\Omega_{\varepsilon}:=\Big\{x:\varrho_{\varepsilon}x+x_{\varepsilon}\in \Omega\Big\}$. In light of Lemma \ref{regularizationsameminimizerl}, we obtain that 
\begin{equation}\label{eq3.4_3}
   \big{\vert} \varrho_{\varepsilon}^{\gamma'}\lambda_{\varepsilon}\big{\vert}\leq C, \ \text{for some}\ C>0. 
\end{equation}

We apply \eqref{InfityRiesz} with $s=1$ and $t=\bar q$, 
then Sobolev interpolation inequality with $\theta_0=\frac{q_{\alpha}}{\bar{q}}\in (0,1)$ to obtain that there exist positive constants $C_1,$ $C_2$ and $C_3$ such that
\begin{align}\label{eq3.4_4}
  \Vert K_{\alpha}\ast m_{\varepsilon}\Vert_{L^{\infty}(\Omega)}
    &\leq C_1 \Vert m_{\varepsilon}\Vert_{L^{1}(\Omega)} +C_2\Vert m_{\varepsilon}\Vert_{L^{\bar q}(\Omega)} \nonumber \\
   &\leq C_1 +C_2\left(  \Vert m_{\varepsilon}\Vert^{\theta_0}_{L^{q_{\alpha}}(\Omega)}\Vert m_{\varepsilon}\Vert^{1-\theta_0}_{L^{\infty}(\Omega)}  \right)  \\
   &\leq C_{1}+C_{3}M_{\varepsilon}^{1-\frac{q_{\alpha}}{\bar{q}}},\nonumber
\end{align}
where we have used \eqref{sameminimizer_22} and the fact that $q_{\alpha}<\bar{q}<\infty$ for every $\alpha<n$. Since $m_{\varepsilon}(x)\leq M_{\varepsilon}$ for any $x\in \Omega$, we get from Young's inequality for convolution, the properties of mollifiers and \eqref{eq3.4_4} that 
\begin{align}\label{eq3.4_4_0}
\varrho_{\varepsilon}^{\gamma'}\Vert\left( K_{\alpha}\ast {m}_{\varepsilon} \ast \eta_{{\varepsilon}}\right)(x_{\varepsilon}+\varrho_{\varepsilon} x)\Vert_{L^\infty(\Omega_{\varepsilon})}=&\varrho_{\varepsilon}^{\gamma'}\Vert K_{\alpha}\ast m_{\varepsilon} \ast \eta_{\varepsilon}\Vert_{L^{\infty}(\Omega)} \nonumber\\
\leq& \varrho_{\varepsilon}^{\gamma'}\Vert K_{\alpha}\ast m_{\varepsilon}\Vert_{L^{\infty}(\Omega)}\\
 \leq& C_1 \varrho_{\varepsilon}^{\gamma'}+C_3 \varrho_{\varepsilon}^{\gamma'} M_{\varepsilon}^{1-\frac{q_{\alpha}}{\bar{q}}}=C_1 \varrho_{\varepsilon}^{\gamma'}+C_3  M_{\varepsilon}^{1-\frac{q_{\alpha}}{\bar{q}}-r\gamma'} \leq C, \nonumber
\end{align} 
where the last inequality holds since $1-\frac{q_{\alpha}}{\bar{q}}< r\gamma'$, which follows from \eqref{eq3.4_0}.


Next, we decompose our arguments into the following two cases: 

\textbf{Case 1:}  Assume that 
\begin{equation}\label{eq3.4_6}
\lim_{\varepsilon\to 0}\frac{d (x_{\varepsilon},\partial\Omega)}{\varrho_{\varepsilon}}=+\infty.
\end{equation}
Then, thanks to \eqref{eq3.4_6}, one has $\Omega_{\varepsilon}\nearrow \mathbb R^n$ as $\varepsilon\to 0$.  Thus, there exists $R>0$ independent of $\varepsilon$ such that $B_{4R}(0)\subset \Omega_{\varepsilon}$ for $\varepsilon>0$ small sufficiently. 
Moreover, by using \eqref{eq3.4_4_0}, we invoke Lemma \ref{HJBexists} to obtain that there exists a constant $\tilde{C}>0$ independent of $\varepsilon$ such  that $$\Vert\nabla\bar{u}_{\varepsilon}\Vert_{L^{\infty}(B_{2R}(0))}\leq \tilde{C}.$$
This, together with $0\leq \bar m_{\varepsilon}(x)\leq 1$, yields that 
$$\Vert \bar{m}_{\varepsilon}|\nabla \bar{u}_{\varepsilon}|^{\gamma-2}\nabla \bar{u}_{\varepsilon} \Vert_{L^{\infty}(B_{2R}(0))}\leq \Vert\nabla\bar{u}_{\varepsilon}\Vert^{\gamma'-1}_{L^{\infty}(B_{2R}(0))}\leq {\tilde{C}}^{\gamma'-1}.$$ 
Invoking Lemma \ref{regularity-mw1p-1}, we then obtain that, for $\varepsilon>0$ small enough, $\bar{m}_{\varepsilon}\in W^{1,p}(B_{R}(0))$ for all $p>1$. By choosing $p>n$ and using the Sobolev embedding theorem, one finds $\bar{m}_{\varepsilon}\in C^{0,\theta}(B_{R}(0))$ for some $\theta\in (0,1)$. Since $\bar{m}_{\varepsilon}(0)=1$, by using the  H{\"o}lder's continuity of $\bar{m}_{\varepsilon}$, we have
\begin{equation}\label{eq3.4_7}
  \bar{m}_{\varepsilon}(x)\geq \delta>0\ \text{in}\ \ B_{\tilde{R}}(0),  
\end{equation}
for some $\delta>0$ and $\tilde{R}<R$. Now, 
we reach a contradiction as follows
\begin{equation}\label{eq3.4_9}
 0<C\leq\int_{B_{\tilde{R}}(0)}{\bar{m}_{\varepsilon}^{q_{\alpha}}(x)}\, dx \leq  M_{\varepsilon}^{-q_{\alpha}+rn}\Vert m_{\varepsilon}\Vert_{L^{q_{\alpha}}(\Omega)}^{q_{\alpha}}\to 0, \ \text{as}\ \varepsilon\to 0, 
\end{equation}
where we have used the fact that $-q_{\alpha}+rn<0$ for all $\alpha_{sc}<\alpha<n$ due to \eqref{eq_r}. 

 \textbf{Case 2:} Assume that there exists a constant $C>0$ independent of $\varepsilon$ such that 
 \begin{equation}\label{eq3.4_92}
 \lim_{\varepsilon\to 0}\frac{d (x_{\varepsilon},\partial\Omega)}{\varrho_{\varepsilon}}\leq C<+\infty.
 \end{equation}
 Then, up to subsequence, we may assume that $x_{\varepsilon}\to x_{0}\in \partial\Omega$ as $\varepsilon \to 0$. Without loss of generality, we may suppose $x_{0}=0\in \partial\Omega$ and $\boldsymbol{n}(0)=-e_n$. Let $x'=(x_1,x_2,\cdots,x_{n-1})\in \mathbb R^{n-1}$. Since $\partial\Omega$ is of class $C^{3}$, one has there exists $U\in \mathbb R^{n}$, $\Gamma\in \mathbb R^{n-1}$ and $\phi(x')\in C^{2,\theta}(\Gamma)$ for some $\theta\in(0,1)$ such that
\begin{equation}\label{straighten-1}
\phi(0)=0,\ \ \ \ \nabla\phi(0)=0,
\end{equation}
\begin{equation}\label{straighten-2}
\partial\Omega\cap U=\left\{(x',x_{n}):x_{n}=\phi(x')\right\} \ \text{and}\ \ \Omega\cap U=\left\{(x',x_{n}):x_{n}>\phi(x')\right\}.
\end{equation}
Then, we locally straighten the boundary with a
diffeomorphism $\Phi: \mathbb R^{n}\to \mathbb R^{n}$ and define 
\begin{align}\label{straighten-3}
y_{i}=\Phi^{i}(x)=\left\{\begin{array}{ll}
x_{i}-x_{n}\frac{\partial\phi}{\partial x_{i}}(x'),\ \ \text{for}\ \ 1\leq i\leq n-1, \\
\vspace{0.5ex}
x_{n}-\phi(x'),\ \ \ \ \  \ \  \text{for}\ \ i=n.
\end{array}
\right.
\end{align}
It is straightforward to show that $\Phi$ is invertible in a neighborhood of $0\in\partial \Omega$. Next, we extend the solution by introducing the following even reflection 
\begin{equation}\label{straighten-4}
   \tilde{m}_{\varepsilon}(y)= \bar m_{\varepsilon}\left(\frac{\Phi^{-1}(y',|y_{n}|)-x_{\varepsilon}}{\varrho_{\varepsilon}}\right),
\end{equation}
and
\begin{equation}\label{straighten-5}
    \tilde{u}_{\varepsilon}(y)=\bar u_{\varepsilon}\left(\frac{\Phi^{-1}(y',|y_{n}|)-x_{\varepsilon}}{\varrho_{\varepsilon}}\right),
\end{equation}
where $y'=(y_1,y_2,\cdots,y_{n-1})\in \mathbb R^{n-1}$. One can deduce from the Neumann boundary conditions that $\frac{\partial\tilde{m}_{\varepsilon}}{\partial y_{n}}\Big {|}_{\{y_n=0\}}=0$.  Moreover, by a direct calculation, we obtain that $\tilde{m}_{\varepsilon}$ and $\tilde{u}_{\varepsilon}$ satisfy equation \eqref{eq3.4_2} in a fixed neighborhood of the boundary point that independent of $\varepsilon$.  Then, we proceed with the same argument shown in Case 1 and obtain a contradiction.

Now, we focus on the case of $\alpha=\alpha_{sc}>0$. By using the properties of convolution and \eqref{H-L-S}, one has 
\begin{equation}\label{sobolevcritical}
    C_{f}\Vert K_{\alpha_{sc}}\ast m_{\varepsilon}\ast \eta_{\varepsilon}\Vert_{L^{\frac{n}{\gamma'}}(\Omega)}\leq C_{f}\Vert K_{\alpha_{sc}}\ast m_{\varepsilon}\Vert_{L^{\frac{n}{\gamma'}}(\Omega)}\leq  C_{f}C_{n,\gamma'} \Vert m_{\varepsilon}\Vert_{L^{\frac{n}{n-\gamma'}}(\Omega)},
\end{equation}
where $C_{n,\gamma'}>0$ is a constant depending only on $n$ and $\gamma'$. Thanks to Lemma \ref{lobalcriticlemma}, we further obtain 
\begin{equation}\label{sobolevcritical_1}
    \|m_{\varepsilon}\|_{L^{q_{\alpha_{sc}}}(\Omega)}=\|m_{\varepsilon}\|_{L^{\frac{n}{n-\gamma'}}(\Omega)}\leq \bar{r}=C_{\alpha}+1,
\end{equation}
which, together with \eqref{sobolevcritical}, implies that
\begin{equation}\label{sobolevcritical_2}
     C_{f}\Vert K_{\alpha_{sc}}\ast m_{\varepsilon}\ast \eta_{\varepsilon}\Vert_{L^{\frac{n}{\gamma'}}(\Omega)}\leq  C_{f} C_{n,\gamma'} \bar{r}.
\end{equation}
Thus, one finds that $ C_{f} \Vert K_{\alpha_{sc}}\ast m_{\varepsilon} \ast \eta_{\varepsilon}\Vert_{L^{\frac{n}{\gamma'}}(\Omega)}\leq \sigma$  whenever $ C_{f}\leq \frac{\sigma}{\bar{r}C_{n,\gamma'}}$ for some sufficiently small $\sigma>0$ . Furthermore, by applying Lemma \ref{nablabarrier_1}, we can obtain
 \begin{equation}\label{sobolevcritical_3}
        \|\nabla u_{\varepsilon}\|_{L^{n(\gamma-1)}(\Omega)}\leq \left(\frac{\zeta}{\gamma C_{H}}\right)^{\frac{1}{\gamma-1}},
\end{equation}
where $C_{\Lambda}$ is given in \eqref{2.6-2} and $\zeta$ is shown in Lemma \ref{b-infty}.  Therefore, one gets
 \begin{equation}\label{sobolevcritical_4}
    \Vert C_{H}\gamma |\nabla u_{\varepsilon}|^{\gamma-2}\nabla u_{\varepsilon} \Vert_{L^{n}(\Omega)}\leq \zeta. 
\end{equation}
Invoking Lemma \ref{b-infty}, we obtain that $m_{\varepsilon}$ is uniformly bounded in $W^{1,p}(\Omega)$ for any $p<n$ and thus $m_{\varepsilon}$ is uniformly bounded in $L^{p^{\ast}}(\Omega)$. Here $p^{\ast}=\frac{n p}{n-p}$ is chosen such that $p^{\ast}>q_{\alpha_{sc}}=\frac{n}{n-\gamma'}$. From this, by replacing $q_{\alpha}$ with $p^\ast$ and repeating above arguments, one can show $m_{\varepsilon}$ is uniformly bounded in $ L^{\infty}(\Omega)$. 
\end{proof} 
 Lemma \ref{mepsinfty} demonstrates that when $\alpha=\alpha_{sc}$, the uniform boundedness of $m_{\varepsilon}$ in $L^\infty$ is established only when $C_f$ is small enough since it is necessary to use Lemma \ref{b-infty} to find  uniform bounds for $m_{\varepsilon}$ in some $L^{p}$ with $p>q_{\alpha_{sc}}$.  Whereas, when $\alpha>\alpha_{sc}$,
the uniform boundedness of $m_{\varepsilon}$ in general holds.  Now, we are ready to prove Theorem \ref{thm11-optimal}, which is
 
\medskip

\noindent \textbf{Proof of Theorem \ref{thm11-optimal}:}
Let $(u_{\varepsilon},m_{\varepsilon},\lambda_{\varepsilon})$ be the solutions of \eqref{regularizedgoalmodel} obtained in Lemma \ref{regularizationsameminimizerl}.  Namely, $(u_{\varepsilon},m_{\varepsilon},\lambda_{\varepsilon})$ satisfies 
 \begin{align}\label{proofTHM_1}
\left\{\begin{array}{ll}
-\Delta u_{\varepsilon}+C_H|\nabla u_{\varepsilon}|^{\gamma}+\lambda_{\varepsilon}=- C_{f}\left( K_{\alpha}\ast m_{\varepsilon} \ast \eta_{\varepsilon}\right),\  &\text{in} \ \Omega,\\
\vspace{0.5ex}
\Delta m+C_{H}\gamma\nabla\cdot (m_{\varepsilon}|\nabla u_{\varepsilon}|^{\gamma-2}\nabla u_{\varepsilon})=0,\  &\text{in} \ \Omega,\\
\vspace{0.5ex}
\frac{\partial u_{\varepsilon}}{\partial\boldsymbol{n}}=0,\ \frac{\partial m_{\varepsilon}}{\partial\boldsymbol{n}}+ C_{H}\gamma m_{\varepsilon}|\nabla u_{\varepsilon}|^{\gamma-2}\nabla u_{\varepsilon} \cdot\boldsymbol{n} =0, \ &\text{on} \ \partial\Omega,\\
\vspace{0.5ex}
 \int_{\Omega}m_{\varepsilon}\,dx=1,\ \ \int_{\Omega}u_{\varepsilon}\,dx=0.
\end{array}
\right.
\end{align}
We shall show that, up to subsequence, $\lambda_{\varepsilon}\to \lambda$ in $\mathbb R$, $u_{\varepsilon}\to u$ uniformly in $C^{2}(\overline{\Omega})$ and $m_{\varepsilon}\to m$ in $W^{1,p}(\Omega)$ for every $p>1$, where $(u,m,\lambda)$ is a classical solution of system \eqref{goalmodel}.  In the subsequent analysis, we shall tackle Sobolev subcritical and critical cases, respectively. 

First of all, we consider the case of $\alpha_{sc}<\alpha<n$. Noting the boundedness of $\lambda_\varepsilon$ obtained in \eqref{sameminimizer_2}, we have that, up to a subsequence, $\lambda_\varepsilon \to \lambda$ as $\varepsilon\to 0$. Moreover,
invoking Lemma \ref{mepsinfty}, the properties of convolution and inequality \eqref{InfityRiesz}, we conclude that 
\begin{equation}\label{proofTHM_2}
  \begin{aligned}
\Vert C_{f}\left( K_{\alpha}\ast m_{\varepsilon} \ast \eta_{\varepsilon}\right) \Vert_{L^{\infty}(\Omega)}&\leq  C_{f} \Vert K_{\alpha}\ast m_{\varepsilon}\Vert_{L^{\infty}(\Omega)}\nonumber\\ 
&\leq C \Vert m_{\varepsilon}\Vert_{L^{1}(\Omega)} +C \Vert m_{\varepsilon}\Vert_{L^{\infty}(\Omega)} \\ 
&\leq C+C \Vert m_{\varepsilon}\Vert_{L^{\infty}(\Omega)}\leq C, \nonumber
\end{aligned} 
\end{equation}
for some $C>0$. With the aid of the  $L^{\infty}$ boundedness of $m_{\varepsilon}$  and \eqref{holderRiesz_1}, one further gets $K_{\alpha}\ast m_{\varepsilon}$ in $C^{0,\theta}$ for some $\theta\in(0,1)$. Therefore, by Lemma \ref{HJBexists}, we obtain that 
\begin{equation}\label{proofTHM_2_1}
\Vert\nabla u_{\varepsilon}\Vert_{L^{\infty}(\Omega)}\leq C_{1} \ \text{for some} \ \ C_{1}>0 \ \text{independent of}\
\varepsilon>0. 
\end{equation}
Now, we rewrite the Hamilton-Jacobi equation of \eqref{proofTHM_1} as
\begin{equation}\label{proofTHM_3}
    -\Delta u_{\varepsilon}=-C_H|\nabla u_{\varepsilon}|^{\gamma}+h_{\varepsilon}(x)\ \ \text{with}\ \ h_{\varepsilon}(x):=-\lambda_{\varepsilon}-C_{f}\left( K_{\alpha}\ast m_{\varepsilon} \ast \eta_{\varepsilon}\right),\  x\in \Omega.
\end{equation}
By the  boundedness of $m_{\varepsilon}$ in $L^{\infty}(\Omega)$ and \eqref{proofTHM_2_1}, we deduce that $h_{\varepsilon}\in L^{p}(\Omega)$ for all $p>1$. We then apply the standard $W^{2,p}$ estimate into the Hamilton-Jacobi equation and obtain that $\Vert u_{\varepsilon}\Vert_{W^{2,p}(\Omega)}\leq C$ for all $p>1$. This, together with Sobolev embeddings, implies that $u_{\varepsilon}$ are equibounded in  $C^{1,\theta}(\overline{\Omega})$ for some $\theta\in (0,1)$. In addition, since $m_{\varepsilon}\in L^{\infty}(\Omega)$, we can choose a suitable $\tilde{q}>\frac{n}{\alpha}$ such that $m_{\varepsilon}\ast \eta_{\varepsilon}\in L^{\tilde{q}}(\Omega)$. Then, by using Lemma \ref{HolderforRiesz}, 
we find that $C_{f}\left( K_{\alpha}\ast m_{\varepsilon} \ast \eta_{\varepsilon}\right)\in C^{0,\theta}(\overline{\Omega})$ for some $\theta\in (0,1)$.  Moreover, with the aid of the classical Schauder's estimates, one gets from \eqref{proofTHM_3} that $u_{\varepsilon}$ are equibounded in $C^{2,\theta}(\overline{\Omega})$ for some $\theta\in(0,1)$. Thus, by Arzel{\'a}-Ascoli theorem, we can obtain that, up to a subsequence, $u_{\varepsilon}\to u$ in $C^{2}(\overline{\Omega})$.  Now, we focus on the Fokker-Planck equation in \eqref{proofTHM_1}.  Noting that
\begin{equation}\label{proofTHM_4}
    \Vert C_{H}\gamma |\nabla u_{\varepsilon}|^{\gamma-2}\nabla u_{\varepsilon} \Vert_{L^{\infty}(\Omega)}\leq C,
\end{equation}
one obtains from Lemma \ref{regularity-mw1p} 
that $\Vert m_{\varepsilon}\Vert_{W^{1,p}(\Omega)}\leq C$ for all $p>1$, which combines with Sobolev embedding theorem, implies that $m_{\varepsilon}\in C^{0,\theta}(\overline{\Omega})$ uniformly for some $\theta\in (0,1)$. Hence, we get that, up to a subsequence, $m_{\varepsilon}\to m$ in $W^{1,p}(\Omega)$ for all
$p > 1$. 
Thus, by the bootstrapping argument,we have that $(u,m,\lambda)$ is a classical solution to system \eqref{goalmodel}. Meanwhile,  by using Lemma \ref{regularity-mw1p}, we obtain the positivity of $m$.

Next, we consider the Sobolev-critical case, i.e. $\alpha=\alpha_{sc}$.  By repeating the above arguments, we apply Lemma \ref{mepsinfty} to show the desired conclusion. 

It remains to prove that the solutions we obtained above are minimizers of $\mathcal{E}$. To this end,  we follow the argument shown in \cite{BraidesAndrea} to derive the $\Gamma$-convergence of the regularized energy functional.  We mention that it is not necessary to derive the equicoercivity property of $\mathcal{E}_{\varepsilon}$ since we know a-priori that the sequence of minima converges.  Denoting $X: = L^{q_{\alpha}}(\Omega)\cap W^{1,\beta}(\Omega)\times L^{1}(\Omega)$, we aim to prove that $\mathcal{E}_{\varepsilon}$ $\Gamma$-converges to $\mathcal{E}$ in $X$.  Assume that $(m_{\varepsilon},\boldsymbol{w}_{\varepsilon})\to (m,\boldsymbol{w})$ in $X$, then it follows from the properties of mollifiers that $m_{\varepsilon}\ast \eta_{\varepsilon}\to m$ in $L^{q_{\alpha}}(\Omega)$. Since $\int_{\Omega}\Big|\frac{\boldsymbol{w}}{m}\Big|^{\gamma'}m\,dx$ is semi-continuous and
\begin{equation}\label{proofTHM_4_1}
\lim\limits_{\varepsilon \to 0}\frac{1}{2}\int_{\Omega}C_{f}\bigg\{m_{\varepsilon}(K_{\alpha}\ast m_{\varepsilon})\bigg\}\ast \eta_{\varepsilon}\,dx= \frac{1}{2}\int_{\Omega}C_{f}\bigg\{m(K_{\alpha}\ast m)\bigg\}\, dx,
\end{equation}
we have
\begin{align}\label{proofTHM_41}
\liminf\limits_{\varepsilon \to 0} \mathcal{E}_{\varepsilon}(m_{\varepsilon},\boldsymbol{w}_{\varepsilon})&=\liminf\limits_{\varepsilon \to 0} C_L\int_{\Omega}\Big|\frac{\boldsymbol{w}_{\varepsilon}}{m_{\varepsilon}}\Big|^{\gamma'}m_{\varepsilon}\,dx-\liminf\limits_{\varepsilon \to 0}\frac{1}{2}\int_{\Omega}C_{f}\bigg\{m_{\varepsilon}(K_{\alpha}\ast m_{\varepsilon})\bigg\}\ast \eta_{\varepsilon}\,dx\nonumber\\ 
&\geq C_L\int_{\Omega}\Big|\frac{\boldsymbol{w}}{m}\Big|^{\gamma'}m\,dx-\frac{1}{2}\int_{\Omega}C_{f}\bigg\{m(K_{\alpha}\ast m)\bigg\}\,dx=\mathcal{E}(m,\boldsymbol{w}).
\end{align} 
On the other hand, we get $\mathcal{E}_{\varepsilon}(m_{\varepsilon},\boldsymbol{w}_{\varepsilon})\leq \mathcal{E}_{\varepsilon}(m,\boldsymbol{w})$ by using the fact that $(m_{\varepsilon},\boldsymbol{w}_{\varepsilon})$ is the minimizer of $\mathcal{E}_{\varepsilon}$.  Thanks to \eqref{proofTHM_4_1}, one also has 
$$\mathcal{E}(m,\boldsymbol{w})=\limsup_{\varepsilon\to 0}\mathcal{E}_{\varepsilon}(m,\boldsymbol{w})\geq \limsup_{\varepsilon\to 0} \mathcal{E}_{\varepsilon}(m_{\varepsilon},\boldsymbol{w}_{\varepsilon}).$$
This, together with \eqref{proofTHM_41}, implies that   
$\mathcal{E}_{\varepsilon}(m_{\varepsilon},\boldsymbol{w}_{\varepsilon})\to \mathcal{E}(m,\boldsymbol{w})$.  Finally,  by invoking Lemma \ref{regularizationsameminimizerl}, one finds a minimum $(m_{\varepsilon},\boldsymbol{w}_{\varepsilon})$ of $\mathcal{E}_{\varepsilon}$ yields a solution $(u_{\varepsilon},m_{\varepsilon},\lambda_{\varepsilon})$  of \eqref{regularizedgoalmodel} with $\boldsymbol{w}_{\varepsilon}=-m_{\varepsilon}C_{H}\gamma \vert\nabla u_{\varepsilon}\vert^{\gamma-2}\nabla u_{\varepsilon}$.  In summary, we obtain that the sequence of solutions
$(u_{\varepsilon},m_{\varepsilon},\lambda_{\varepsilon})$ converges in $C^{2}(\overline{\Omega}) \times W^{1,p}(\Omega)\times \mathbb R$ for all $p\geq 1$ to a solution $(u, m, \lambda)$ of the original problem \eqref{goalmodel}. 
 Correspondingly, $(m_{\varepsilon},\boldsymbol{w}_{\varepsilon})$ converges to $(m, -mC_{H}\gamma \vert\nabla u\vert^{\gamma-2}\nabla u)$ in $X$.  Thus,  $(u,m,\lambda)$ solves \eqref{goalmodel} and $(m, -mC_{H}\gamma \vert\nabla u\vert^{\gamma-2}\nabla u)$ is a minimum of $\mathcal{E}$.  This finishes the proof of Theorem \ref{thm11-optimal}. 
$\hfill\qed$  

\section{Discussion}

In this paper, we mainly employ the variational structure to investigate the existence of solutions to the ergodic Mean Field Games system \eqref{goalmodel} with strong focusing nonlocal coupling.  More precisely, concerning the Choquard Coupling, one of typical non-local couplings, we classify the existence of minimizers and obtain that there exist global minimizers in the mass-subcritical and mass-critical cases and local minimizers when the exponent is above the mass-critical one and below or equal to the Sobolev-critical one.  To find the minimizers of $\mathcal{E}$, we introduced a family of regularized energy functionals $\mathcal{E}_{\varepsilon}$, where non-local couplings are smoothed by the standard mollifiers.  Then, we are able to establish the existence of minimizers to regularized problems.  We mention that while showing the existence of local minimizers to $\mathcal{E}$ in the Sobolev-critical case, we imposed smallness conditions on $C_{f}$ to guarantee the regularity of the $m$-component.  Here the Hardy-Littlewood-Sobolev inequality is crucial for controlling the nonlinear coupling in (\ref{goalmodel}).  With the existence of minimizers, we invoked the classical convex duality theory to obtain the existence of classical solutions to regularized Mean-field Game systems. Finally, we derived the uniform $L^\infty$ boundedness of $m$ obtained in regularized Mean-field Game systems via the blow-up argument and took the limit to obtain the existence of solutions to the original problem (\ref{goalmodel}).


There are also some interesting problems that deserve exploration in the future.  It is intriguing to extend our results in the bounded domain into the whole space, where the exponent lies between the mass-supercritical and Sobolev-critical ranges.  In addition, the analysis of solution properties, such as stability, uniqueness, and symmetries, derived from the variational method presents significant challenges, yet remains highly intriguing.


\begin{appendices}
\setcounter{equation}{0}
\renewcommand\theequation{A.\arabic{equation}}
\end{appendices}

\section*{Acknowledgments}

Xiaoyu Zeng is supported by NSFC (Grant Nos. 12322106, 12171379, 12271417). 


\bibliographystyle{abbrv}
\bibliography{ref}

\end{document}